\documentclass[10pt,journal]{IEEEtran}

\usepackage{cite}
\usepackage{color}
\usepackage{amsfonts,amssymb}
\usepackage{bm}
\usepackage{amsmath}
\usepackage{multirow}
\usepackage{algorithmic}
\usepackage{graphicx}
\usepackage{textcomp}
\usepackage{graphicx} 
\usepackage{epstopdf}
\usepackage{array}

\usepackage[thmmarks,amsmath]{ntheorem}
\newtheorem{mythm}{Theorem}
\newtheorem{mydef}{Definition}
\newtheorem{assumption}{Assumption}
\newtheorem{lemma}{Lemma}
\newtheorem{remark}{Remark}

\usepackage{array}
\usepackage{enumerate}
\usepackage{booktabs}
\usepackage{algorithm}
\usepackage{algorithmic}
\usepackage{setspace}
\usepackage{float}
\usepackage{cases}

\qedsymbol{\ensuremath{\square}} 
{   
	\theoremstyle{nonumberplain}
	\theoremheaderfont{\bfseries}
	\theorembodyfont{\normalfont}
	\theoremsymbol{\ensuremath{\square}} 
	\newtheorem{proof}{Proof.}
}
\ifCLASSINFOpdf
\else
\fi
\begin{document}
%
\title{Efficient algorithm for approximating Nash equilibrium
	of distributed aggregative games}
%
%
%

\author{Gehui Xu, Guanpu Chen, Hongsheng Qi, and Yiguang Hong, \IEEEmembership{Fellow, IEEE}
	\thanks{This work was supported by the National Natural Science Foundation of China (No. 61733018, No. 61873262), and by Shanghai Municipal Science and Technology Major Project (No. 2021SHZDZX0100). Corresponding author: Yiguang Hong.}
\thanks{Gehui Xu, Guanpu Chen, and Hongsheng Qi are with Key Laboratory of Systems and Control, Academy of Mathematics and Systems Science, Beijing, China,
	and are also with School of Mathematical Sciences, University of Chinese Academy of Sciences, Beijing, China. (e-mail: xghapple@amss.ac.cn, chengp@amss.ac.cn, qihongsh@amss.ac.cn).}
\thanks{Yiguang Hong is with
	Department of Control Science and Engineering, Shanghai Research Institute for Intelligent Autonomous Systems, Tongji University, Shanghai, and is also with the Key Laboratory of Systems and Control, Academy of Mathematics and Systems Science, Chinese Academy of Sciences, Beijing, China. (e-mail: yghong@iss.ac.cn).}}

\maketitle

\begin{abstract}
	In this paper, we aim to design  a distributed approximate  algorithm for seeking Nash equilibria of an aggregative game. 
	Due to the local set constraints of each player, projection-based
	algorithms have been widely employed for solving such
	problems actually.
	Since it may be quite hard to get the exact projection in practice, we utilize  inscribed polyhedrons to approximate  local set constraints, which yields a related approximate game model.  We first prove that the   Nash equilibrium of the approximate game is the $ \epsilon $-Nash equilibrium of the original game, and then propose a distributed algorithm to seek the $ \epsilon $-Nash equilibrium, where the projection  is then of  a standard form in  quadratic programming. With the help of the existing developed methods for  solving quadratic programming, we show the convergence of the proposed algorithm, and also discuss the computational cost issue related to the approximation. Furthermore, based on the exponential convergence of the algorithm, we estimate the approximation accuracy related to $ \epsilon $.  Additionally, we investigate the computational cost saved by approximation on numerical examples.
\end{abstract}

\begin{IEEEkeywords}
	$ \epsilon $-Nash equilibrium;  Approximation;  Distributed  algorithm; Aggregative game.
\end{IEEEkeywords}

%
\IEEEpeerreviewmaketitle

\section{Introduction}
%
%
%
%
Seeking  Nash equilibria (NE)  in non-cooperative games has been widely investigated in  social sciences and engineering.
As one of the important non-cooperative games, the aggregative game has drawn much growing interest in many fields, such as  
demand response management \cite{ye2016game} and multi-product enterprise oligopoly \cite{nocke2018multiproduct}. Particularly, because of complex topologies, or communication burdens, or privacy issues in large-scale networks, it is of great practical significance to seek NE  in a distributed manner, where players achieve the NE with local data and communications through networks \cite{gharesifard2015price,parise2019distributed,de2019distributed,yi2019asynchronous}.

Since  players' actions are usually constrained by  local sets, projection-based distributed  algorithms for NE or {generalized Nash equilibria (GNE)} seeking have been developed.
For aggregative games, \cite{koshal2016distributed} studied  projected distributed synchronous and asynchronous algorithms for  NE computation over a network, while
\cite{paccagnan2016distributed} investigated a projection-based distributed asymmetric  algorithm  for GNE seeking with affine coupling constraints. Then
\cite{liang2017distributed} designed a projected   distributed continuous-time algorithm for   non-smooth tracking dynamics with coupled constraints. 
Moreover,  \cite{lei2020distributed} proposed a projected    distributed  algorithm  for NE seeking based on iterative Tikhonov regularization methods, while \cite{belgioioso2020distributed} discussed  another projection-based   algorithm on a time-varying communication network for  seeking GNE with partial-decision information.

Various methods are usually adopted for projection operation, such as the  sequential quadratic program (SQP)  \cite{boggs1995sequential}, the interior point method (IPM)  \cite{nesterov1994interior}, and the augmented Lagrangian method (ALM)  \cite{fortin2000augmented}. However,  the computational complexity may be exceptionally high for high-dimensional constraint sets, and the computational error may increase with the expansion of data and model scales. 
On the other hand, hyperplane approximation  was  widely employed in various practical situations such as multi-objective optimization problems \cite{li2019hyperplane}, object tracking in video images \cite{jurie2002hyperplane}, and feature categorization of machine learning \cite{vincent2001k}.  With this inspiring idea, the players' feasible sets are approximated by constructing inscribed polyhedrons, which are thereby enclosed by a series of hyperplanes.   Therefore, it is easier to obtain the projection on the hyperplanes than that on the boundaries of convex sets, because a general projection operation is converted  into a  standard quadratic program, and many developed methods for quadratic programming can be effectively adopted. Although the computational complexity can be   reduced effectively in this way,  the approximate  process inevitably brings the loss of accuracy, related to the discussion of $ \epsilon $-NE. However, considering the applications in distributed computing with large-scale models, it makes  sense to sacrifice a little accuracy  for time saving and complexity reduction.

The motivation of this paper is to explore efficient NE seeking of a distributed aggregative game, where we promote to use inscribed polyhedrons to approximate players' feasible sets.


The main contributions of this paper are listed in the following.
\begin{itemize}
	\item We consider a distributed approximate NE seeking algorithm for aggregative games. Different from those in \cite{belgioioso2020distributed,yi2019asynchronous,lei2020distributed}, we approximate players' local feasible sets with inscribed polyhedrons, which converts the general projection operation  into a standard quadratic program. With the approximation, we study the seeking of an $ \epsilon $-NE of the original game.
	
	\item We discuss the approximation procedure and analyze the approximation.  To be specific, we provide an approximate method for constructing inscribed polyhedrons and discuss the computational cost saved by approximation.  Then we prove that the  NE of the approximate game is the $ \epsilon $-NE of the original game and analyze the factors influencing the accuracy of $ \epsilon $.
	
	\item We  show that the proposed algorithm converges to the $\epsilon$-NE with an exponential rate, and then give  an upper bound of the  value $ \epsilon $. Moreover, we discuss relationships between the computational cost and approximation from different viewpoints.
\end{itemize}

The remainder is organized as follows: Section \ref{s22} provides notations and preliminary knowledge as well as our problem formulation, while Section \ref{s3} discusses the approximation of
players' local feasible sets with inscribed polyhedrons,  and shows a relationship between  the equilibria of the approximate game and the original one. Then Section \ref{s4} obtains the convergence of a distributed approximate algorithm  to seek the NE with treating the projection as a standard quadratic program and  gives an upper bound of the value $\epsilon$, and Section \ref{s6} shows  numerical examples for illustration  of the proposed algorithm. Finally, Section \ref{s7} concludes the paper.

	\section{Aggregative game model}\label{s22}
In this section,  we first give some basic notations  and  preliminary knowledge, and then  formulate our problem.
\subsection{Notations and preliminaries}

Denote $ \mathbb{R}^{n} $(or $ \mathbb{R}^{m\times n} $) as the set of $ n $-dimensional (or $ m $-by-$ n $) real column vectors (or real matrices), and $ I_{n} $ as the $ n\times n $ identity matrix. Let $ \boldsymbol{1}_{n} $(or $ \boldsymbol{0}_{n} $) be the $ n $-dimensional column vector with all elements of $ {1} $ (or $ {0} $). Denote $ A\otimes B $ as the Kronecker product of matrices $ A $ and $ B $. Take $ col(x_{1},\cdots,x_{n}) $$  = (x^{\mathrm{T}}_{1}, \cdots,x^{\mathrm{T}}_{n})^{\mathrm{T}} $and $\|\cdot\|$ as the Euclidean norm of vectors. Denote $ \nabla f $ as the gradient of function $  f $. Denote $\mathbf{B}_{r}(x)\subseteq \mathbb{R}^{n}$ as a ball with the center at point $x$ and the radius $r$.
Moreover, denote $ \mathbf{E}_{v}(c)\subseteq \mathbb{R}^{n} $ as an ellipsoid that 
$$
\sum_{i=1}^{n}\frac{(x_{i}-c_{i})^{2}}{v_{i}^{2}}\leq1,\;  
$$ 
with the center at point $ c\triangleq(c_{1},\cdots,c_{n}) $ and the semiaxis  $v\triangleq(v_{1},\cdots,v_{n})$.

A set $ K \subseteq \mathbb{R}^{n} $ is convex if $ \omega x_{1}+(1-\omega)x_{2} \in K$ for  any
$ x_{1}, x_{2}\in K $ and $ 0\leq\omega\leq 1   $. For a closed convex set $ K $, the projection map $ \Pi_{K}: \mathbb{R}^{n} \rightarrow K $ is defined as
$$
\Pi_{K}(x) \triangleq \underset{y \in K}{\operatorname{argmin}}\|x-y\| .
$$
The following basic property hold:
$$
\|\Pi_{K}(x)-\Pi_{K}(y)\| \leq \|x-y\|,\quad \forall x, y \in \mathbb{R}^{n}.
$$
A function $ f:\mathbb{R}^{n} \rightarrow \mathbb{R} $ is convex on $ K $ if 
$$ f (\omega x_{1}+(1-\omega)x_{2})\leq \omega f(x_{1})+(1-\omega)f(x_{2}), $$ 
for any $ x_{1} $, $ x_{2}\in K $ and $ 0\leq \omega\leq1 $.

A mapping $ F : \mathbb{R}^{n} \rightarrow \mathbb{R}^{n} $ is said to be $ \kappa $-strongly monotone on a set $ D $ if there exists a constant $ \kappa > 0 $  such that
$$
(F(x)-F(y))^{\mathrm{T}}(x-y) \geq \kappa\|x-y\|^{2}, \quad \forall x, y \in D.
$$
Given a set $ D \subseteq \mathbb{R}^{n} $ and a map $ F: D \rightarrow \mathbb{R}^{n} $, the variational inequality problem $\mathrm{VI}(D, F)$ is defined to find a vector $ x^{*}\in D  $ such that $$
\left(y-x^{*}\right)^{\mathrm{T}} F\left(x^{*}\right) \geq 0, \quad \forall y \in D,
$$
whose  solution  is denoted by $\mathrm{SOL}(D, F)$.

The following lemma shows an equivalent relationship between the solution of $\mathrm{VI}(D, F)$ and the projection map, and discusses the properties of the solution of $\mathrm{VI}(D, F)$. Readers can find more details in \cite[Proposition 1.5.8, Corollary 2.2.5, and Theorem 2.3.3]{facchinei2007finite}.
\begin{lemma}\label{p1}
	Consider $\mathrm{VI}(D, F)$, where the set $ D \subseteq \mathbb{R}^{n} $ is convex and the map $ F : D \rightarrow \mathbb{R}^{n} $ is continuous. The following statements hold:
	\begin{enumerate}[(1)]
		\item $
		x \in \operatorname{SOL}(D, F) \Leftrightarrow x=\Pi_{D}(x- \theta F(x)), \forall \theta>0
		$;
		\item if $ D $ is compact, then $\mathrm{SOL}(D, F)$ is nonempty and compact;
		\item if $ D $ is closed and $ F(x) $ is strongly monotone, then $\mathrm{VI}(D, F)$ has at most one solution.
	\end{enumerate}	
\end{lemma}
Take $ X,\, Z\subseteq  \mathbb{R}^{n} $ as two non-empty sets. For $ y \in \mathbb{R}^{n}  $,  denote $ dist(y,Z) $ as the distance between $ y $ and $ Z $, i.e.,
$$  dist(y,Z) = \inf \limits_{z\in Z}\|y- z\|.  $$
Define the Hausdorff metric of $ X,Z\subseteq  \mathbb{R}^{n} $ by $$ H(X,Z) = \max\{\sup\limits_{x\in X}dist(x,Z),\sup\limits_{z\in Z}dist(z,X) \}.$$The Hausdorff metric integrates all compact sets into a metric space.

A directed graph  is defined as $ \mathcal{G} = (\mathcal{I},\mathcal{E}) $
with the node set  $ \mathcal{I} =\{1,2,\cdots,N\} $ and the edge set $ \mathcal{E} $. $ A = [a_{i j}] \in \mathbb{R}^{n\times n} $ is the adjacency matrix of $ \mathcal{G} $ such that if $ (j, i) \in \mathcal{E}  $, then $ a_{i j}> 0 $,   which
means that $j$ belongs to  $ i $'s neighbor set 
and   $ i $ can receive the message sent from agent $ j $, and $ a_{i j}= 0 $ otherwise. A graph is said to be strongly connected if there is a sequence of intermediate vertices connected by edges for any pair of vertices.
A graph is weight-balanced if $\sum_{j=1}^{N} a_{i j}=\sum_{j=1}^{N} a_{ji}$ for every $i \in \mathcal{I} $. The Laplacian matrix is $ L = \Delta-A $, where $\Delta=\operatorname{diag}\left\{d_{1}, \ldots, d_{N}\right\} \in \mathbb{R}^{N \times N}$ with $ d_{i}=\sum_{j=1}^{N} a_{i j} $.

The following lemma is about the Laplacian matrix \cite{bullo2009distributed}.
\begin{lemma}\label{p4}
	Considering a directed graph $ \mathcal{G} $,	
	\begin{enumerate}[(1)]
		\item  $ \mathcal{G} $ is weight-balanced if and only if $ L+ L^\mathrm{T}$ is positive semidefinite;
		\item  $ \mathcal{G} $ is strongly connected if and only if zero is a simple eigenvalue of $  L $.
	\end{enumerate}
\end{lemma}

\subsection{Problem Formulation}
Consider {an} $N$-player aggregative game,  where the players are indexed  by $\mathcal{I} = \{1,\cdots,N\}$. For each $i \in \mathcal{I}$, the $ i $th player has an action variable $ x_{i} $ in a local feasible set $  \Omega_{i}\subseteq \mathbb{R}^{n} $. Denote $   \boldsymbol{\Omega}\triangleq\prod_{i=1}^{N}\Omega_{i} \subseteq \mathbb{R}^{nN} $, $ \boldsymbol{x}\triangleq col\{x_{1}, . . . ,x_{N}\} \in \boldsymbol{\Omega} $ as the action profile for all players, and $ \boldsymbol{x}_{-i}\triangleq col\{x_{1}, . . . ,x_{i-1}, x_{i+1}, . . . ,x_{N}\}$ as the action profile for all players except player $ i $.

The $ i $th player has a payoff function $ J_{i}(x_{i},\boldsymbol{x}_{-i}) : \mathbb{R}^{nN}\rightarrow \mathbb{R} $.  Define an  aggregative term   as
$$\mathcal{Q}(\boldsymbol{x}) \triangleq \frac{1}{N} \sum_{i=1}^{N} q_{i}\left(x_{i}\right).$$ Here $ q_{i}: \mathbb{R}^{n}\rightarrow \mathbb{R}^{M} $ is a map for the local contribution to the aggregation. Specifically,  $ J_{i}\left(x_{i}, \boldsymbol{x}_{-i}\right)=f_{i}(x_{i}, \mathcal{Q}(\boldsymbol{x}))  $ with the function $f_{i}: \mathbb{R}^{n+M}\rightarrow \mathbb{R}$. Given $ \boldsymbol{x}_{-i} $, the $ i $th player intends to solve
\begin{equation}\label{f1}
	\min \limits_{x_{i} \in \Omega_{i}} J_{i}\left(x_{i}, \boldsymbol{x}_{-i}\right).
\end{equation}
\begin{mydef}[Nash equilibrium]\label{d1}
	A profile $ \boldsymbol{x}^{*} $ is said to be a Nash equilibrium (NE) of game (\ref{f1}) if
	\begin{equation*}\label{e1}
		J_{i}\left(x_{i}^{*}, \boldsymbol{x}_{-i}^{*}\right) \leq J_{i}\left(x_{i}, \boldsymbol{x}_{-i}^{*}\right),\;\forall i \in \mathcal{I},\;\forall x_{i}\in \Omega_{i}.
	\end{equation*}
\end{mydef}

In reality with various uncertainties,  NE may not exist or be easily calculated. Therefore, we introduce the following  definition.
\begin{mydef}[$ \epsilon $-Nash equilibrium]\label{d2}
	A profile $ \boldsymbol{x}^{*} $ is said to be an $ \epsilon $-Nash equilibrium of game (\ref{f1})  if
	\begin{equation}\label{e6}
		J_{i}\left(x_{i}^{*}, \boldsymbol{x}_{-i}^{*}\right) \leq J_{i}\left(x_{i}, \boldsymbol{x}_{-i}^{*}\right)+\epsilon,\;\forall i \in \mathcal{I},\;\forall x_{i}\in \Omega_{i},\;
	\end{equation}
	where the constant $ \epsilon>0 $. Particularly, $ \boldsymbol{x}^{*} $  is said to be a NE when $ \epsilon= 0 $.
\end{mydef}

 The local payoff functions $J_{i}$, set constraints $\Omega_{i}$, and decision
variables $ x_{i} $ are the private  information. Moreover, the aggregative term $ \mathcal{Q}(\boldsymbol{x})  $  contains all the players' decisions, which cannot be observed by each player directly. Thus,
player $ i $ generates an estimate $ \zeta_i $ of this aggregative term and exchange this information with its local neighbors through a network $\mathcal{G} $.

For clarification, we denote the pseudo-gradient by $$F(\boldsymbol{x}) \triangleq col\left\{\nabla_{x_{1}} J_{1}\left(\cdot, \boldsymbol{x}_{-1}\right), \ldots, \nabla_{x_{N}} J_{N}\left(\cdot, \boldsymbol{x}_{-N}\right)\right\}.$$
Define the map $ U_{i}: \mathbb{R}^{n}\times\mathbb{R}^{M} \rightarrow \mathbb{R}$ as
\begin{align}\label{aaa}
	U_{i}\left(x_{i}, \zeta_{i}\right) &\left.\triangleq \nabla_{\boldsymbol{x}_{i}} J_{i}\left(\cdot, \boldsymbol{x}_{-i}\right)\right|_{\mathcal{Q}(\boldsymbol{x})=\zeta_{i}} \\
	&=\left.\left(\nabla_{x_{i}} f_{i}(\cdot, \mathcal{Q})+\frac{1}{N} \nabla_{\mathcal{Q}} f_{i}\left(x_{i}, \cdot\right)^{\mathrm{T}} \nabla q_{i}\right)\right|_{\mathcal{Q}=\zeta_{i}}.\notag
\end{align}
Let $$U(\boldsymbol{x}, \boldsymbol{\zeta}) \triangleq col\left(U_{1}\left(x_{1}, \zeta_{1}\right),\ldots, U_{N}\left(x_{N}, \zeta_{N}\right)\right).$$
Obviously, $U\left(\boldsymbol{x}, \mathbf{1}_{N} \otimes \mathcal{Q}(\boldsymbol{x})\right)=F(\boldsymbol{x})$.

We give the following assumptions for game (\ref{f1}).

\begin{assumption}
	$\ $
	\begin{itemize}
		\item For $ i \in \mathcal{I} $, $ \Omega_{i} $ is compact and convex.	
		\item For $ i \in \mathcal{I} $, the payoff function $ J_{i}(\cdot) $ is Lipschitz continuous in $ \boldsymbol{x} $, while $ J_{i}(\cdot,\boldsymbol{x}_{-i}) $  and the map $ q_{i}(x_{i}) $ are   continuously differentiable in $  x_{i} $. Moreover, the pseudo-gradient $F(\boldsymbol{x}) $ is $ \kappa $-strongly monotone on the set $ \boldsymbol{\Omega}  $.
		\item  The map $ U(\boldsymbol{x}, \boldsymbol{\zeta}) $ is $ c_{1}$-Lipschitz continuous in $ \boldsymbol{x}\in  \boldsymbol{\Omega}$  and  $ c_{2}$-Lipschitz continuous in $ \boldsymbol{\zeta} $ for some constants $ c_{1}$, $ c_{2}>0$. Besides, for  $ i \in \mathcal{I} $, $ q_{i} $ is $ c_{3}$-Lipschitz continuous on $ \Omega_{i} $ for a constant $ c_{3}>0$.
		\item The communication network $ \mathcal{G} $ is  strongly connected and weight-balanced.	
	\end{itemize}
\end{assumption}

The  assumptions of convexity and differentiability about payoff functions are quite common and {have been} widely used in the literature.  
Besides, the strong monotonicity of the pseudo-gradient map $ F $ has been widely adopted to guarantee the uniqueness of  NE  \cite{yi2019operator,zhang2019distributed11,lu2018distributed}. Additionally, the assumption about  the Lipschitz continuity of  $ U(\boldsymbol{x}, \boldsymbol{\zeta}) $ and  $ q_i$ is the same as that given in  \cite{liang2019distributed}. Moreover, the strongly connected and weight-balanced digraph is a generalization of connected undirected graphs
in   \cite{deng2019distributed,tang2018distributed}, and is also employed in some other distributed algorithms \cite{liang2020distributed,deng2018distributed}.

The following lemma reveals the relationship of a NE  $ \boldsymbol{x}^{*} $ and a solution to $ \mathrm{VI}(\boldsymbol{\Omega}, F(\boldsymbol{x}))$, referring to \cite[Proposition 1.4.2]{facchinei2007finite} and Lemma \ref{p1}.

\begin{lemma}\label{p3}
	Under Assumption 1, a profile $ \boldsymbol{x}^{*} $ is a NE if and only if $$ \boldsymbol{x}^{*} \in \mathrm{SOL}(\boldsymbol{\Omega}, F(\boldsymbol{x})).$$ Moreover, the  game (\ref{f1}) admits a unique  Nash equilibrium  $ \boldsymbol{x}^{*} $.		
\end{lemma}

Therefore, the main task of this paper is to design a distributed algorithm for seeking a NE of the  aggregative game (\ref{f1}). 
Due to players' local feasible sets,  projection-based methods have been widely used to solve related problems in the literature, e.g.
\cite{belgioioso2020distributed,yi2019asynchronous,lei2020distributed}.
Sometimes, it is  not so easy  to obtain the exact projection points  in practice.
In the following sections, we  provide a scheme to reduce the complexity with an approximate solution. 

	\section{Problem approximation}\label{s3}

As we know, it is always easier to obtain the projection points on the hyperplanes than on the boundaries of general set constraints. Therefore, in this section,  we use inscribed polyhedrons to approximate the players' local feasible sets.

An inscribed polyhedron of a closed convex set is defined as a polyhedron with  all its vertices  on the boundary of the convex set. These vertices construct a series of hyperplanes naturally, which enclose an inscribed polyhedron. Denote  $  \boldsymbol{\mathcal{D}}_{s}  =\prod_{i=1}^{N} \mathcal{D}^{i}_{s_{i}} $, where $ \mathcal{D}^{i}_{s_{i}} $ is an inscribed polyhedron of $ \Omega_{i} $ with $ s_{i} $ vertices, expressed as
\begin{equation}\label{bbb}
	\mathcal{D}_{s_{i}}^{i}=\left\{x_{i} \in \mathbb{R}^{n}: B^{i} x_{i} \leq b^{i}\right\}.
\end{equation}
{Here $ B^{i} \in \mathbb{R}^{p_{i}\times n}  $
	represent  normal vectors of the hyperplanes enclosing  $ \mathcal{D}_{s_{i}}^{i} $ with normalized rows,   $ b_{i} \in \mathbb{R}^{p_{i}} $  are the distances from the hyperplanes to the origin point, and  $ p_{i} $ is the number of hyperplanes for $ i\in \mathcal{I} $.}

The approximation of convex sets by inscribed polyhedrons has been studied in different problems \cite{dudley1974metric,bronshtein1976varepsilon}, which
can  be explicitly expressed by  linear inequalities.
Here our approximation of inscribed polyhedrons  concentrates on players' local feasible sets,  different from the   approximate view angles \cite{ lou2014approximate}, and the  approximation for system parameters \cite{chen2021distributed}.
In fact, the approximate process  with  inscribed polyhedrons makes projection on a polyhedron easier than  directly on a general set,  because 
the projection of point $ x_{0} $ on  a hyperplane $ D=\left\{x | B_{0}^{\mathrm{T}} x = b_{0}\right\} $ can be written explicitly as
$ \Pi_{D}(x_{0})=x_{0} +(b_{0}-B_{0}^{\mathrm{T}} x_{0})B_{0}/\|B_{0}\|^{2}$, 	
which can save the corresponding computational cost.

Thereby, with the help of  inscribed polyhedrons, we consider a related approximate game,
\begin{equation}\label{e9}
	\begin{array}{l}
		\min \limits_{x_{i}\in \mathcal{D}_{s_{i}}^{i} } J_{i}\left(x_{i}, \boldsymbol{x}_{-i}\right) 	.
	\end{array}
\end{equation}

Before revealing the relationship between the approximate game (\ref{e9}) and the original game (\ref{f1}), we first discuss how the Hausdorff distance between two different inscribed polyhedrons influences the relationship of the normal vectors of their hyperplanes. Denote  $\mathcal{D}_{s_{1}}^{1}$ as  an inscribed polyhedron of  a convex and compact set $ \Omega \subseteq \mathbb{R}^{n} $ with $ W_{1} $ as the set of vertices on the boundary of $ \Omega $, i.e.,
\begin{equation}\label{s1}	 \mathcal{D}_{s_{1}}^{1}=\left\{x \in \mathbb{R}^{n}: B^{1} x \leq b^{1}\right\}.
\end{equation}
Similarly, denote $\mathcal{D}_{s_{2}}^{2}$ as another inscribed polyhedron with $ W_{2} $ as the set of vertices on the boundary of $ \Omega $, where $ W_{2}= W_{1}\cup \{w_{0}\} $ with $ w_{0} $ as an additional vertex, i.e.,
\begin{equation}\label{s2} \mathcal{D}_{s_{2}}^{2}=\left\{x \in \mathbb{R}^{n}: B^{2} x \leq b^{2}\right\}.
\end{equation}
Suppose that there are $ p_{1} $ rows of $ B^{1} $ and $ b^{1} $, $ p_{2} $ rows of $ B^{2} $ and $ b^{2} $, the first $ p_{1}-1 $ rows of $ B^{1} $ are the same as the first $ p_{1}-1 $ rows of $ B^{2} $. As a result, the two matrices can be written by row as
\begin{equation}\label{bb}
	B^{1}=\left[\begin{array}{c}B_{1} \\ \vdots \\ B_{p_{1}-1} \\ B_{p_{1}}^{1}\end{array}\right], \quad B^{2}=\left[\begin{array}{c}B_{1} \\ \vdots \\ B_{p_{1}-1} \\ B_{p_{1}}^{2} \\ \vdots \\ B_{p_{2}}^{2}\end{array}\right] .
\end{equation}
%
\begin{lemma}\label{l1}
	For $ B^{2}_{i} $ as any row of matrix $ B^{2} $,  there exists   a  corresponding row  $ B^{1}_{j(i)} $ of matrix $ B^{1} $ such that  	
	$$	
	\left\|B_{i}^{2}-B_{j(i)}^{1}\right\| \rightarrow0, \quad\mathrm{as}\; H(\mathcal{D}_{s_{1}}^{1},\mathcal{D}_{s_{2}}^{2}) \rightarrow 0.$$
\end{lemma}
The proof of  Lemma \ref{l1} can be found in Appendix \ref{a11} In addition, the following lemma describes  the Hausdorff metric between a convex set and its inscribed polyhedron, referring to \cite{2008Approximation}.
\begin{lemma}
	\label{l2}
	For a  convex set $\Omega \subseteq \mathbb{R}^n $, there exists an  inscribed polyhedron $ \mathcal{D}_{s} $ of $\Omega$ such that the upper bound of the Hausdorff metric between $\Omega$ and $ \mathcal{D}_{s} $ satisfies $$ H(\mathcal{D}_{s} , \Omega) \leq \frac{C_{\Omega}}{s^{2/(n-1)}}, $$
	where $ C_{\Omega} $ is a constant related with the curvature of $ \Omega $ and $ s $ is the number of vertices in $ \mathcal{D}_{s} $.
\end{lemma}

		{Based on Lemmas \ref{l1} and  \ref{l2}, it is time to  reveal the  relationship between the			approximate game (\ref{e9})  and the original game (\ref{f1}).}
Note that  the Nash equilibrium $ \boldsymbol{x}^{*} $ of game (\ref{e9}) is the unique solution to $\mathrm{VI}(\boldsymbol{\mathcal{D}}_{s}, F(\boldsymbol{x}))$ by Lemma \ref{p3}.
If the payoff function $ J_{i} $ is fixed, then  different polyhedron approximations  result in different variational inequality solutions. Thereby, we  write $ \boldsymbol{x}^{*} = \boldsymbol{x}^{*}(\boldsymbol{\mathcal{D}}_{s}) $ for game (\ref{e9}). Moreover,  denote   the unique  Nash equilibrium by $  \boldsymbol{x}^{*}(\boldsymbol{\Omega}) $ for  game (\ref{f1}). Then we have the following result.
\begin{mythm}\label{t1}
	Under Assumption 1, the NE of  the approximate game (\ref{e9})  is the $ \epsilon $-NE of the original game (\ref{f1}).
\end{mythm}
\begin{proof}
Take
$$
\boldsymbol{\mathcal{D}}_{s_{1}}=\prod_{i=1}^{N} \mathcal{D}_{s_{1, i}}^{i}, \quad \boldsymbol{\mathcal{D}}_{s_{2}}=\prod_{i=1}^{N} \mathcal{D}_{s_{2, i}}^{i}
$$
as two arbitrarily inscribed polyhedrons of $ \boldsymbol{\Omega} $.  Denote $\boldsymbol{\mathcal{D}}_{s_{1}+N}=\prod_{i=1}^{N} \mathcal{D}_{s_{1, i}+1}^{i}$, where vertices  in $ \mathcal{D}_{s_{1, i}+1}^{i} $ consist of all nodes in $ \mathcal{D}_{s_{1, i}}^{i} $ and one different vertex in $ \mathcal{D}_{s_{2, i}}^{i} $ for $ i\in \mathcal{I} $. $ \boldsymbol{\mathcal{D}}_{s_{1}+2N}$, $ \boldsymbol{\mathcal{D}}_{s_{1}+3N}$,  $ \cdots $, $ \boldsymbol{\mathcal{D}}_{s_{1}+s_{2}} $ are denoted in a similar way, where $\boldsymbol{\mathcal{D}}_{s_{1}+s_{2}}=\prod_{i=1}^{N}\mathcal{D}^{i}_{s_{1, i}+s_{2, i}} $ is the profile
of polyhedrons whose vertices consist of all the vertices
in both $ \boldsymbol{\mathcal{D}}_{s_{1}} $ and $ \boldsymbol{\mathcal{D}}_{s_{2}} $.
{Without losing generality, consider $ s_{2, i}\leq s_{2, j}$. If  $ s_{2, i}< s_{2, j}$	and there is no additive point in $ \mathcal{D}_{s_{2, i}}^{i} $ for $ \mathcal{D}^{i}_{s_{1, i}+s_{2, i}} $, then we keep $ \mathcal{D}^{i}_{s_{1, i}+s_{2, i}} $ unchanged  and  increase the vertices in $ \mathcal{D}_{s_{2, j}}^{j} $ successively. Continue this process until  $ \boldsymbol{\mathcal{D}}_{s_{1}+s_{2}} $ is reached.} $ \boldsymbol{\mathcal{D}}_{s_{2}+N}$, $ \boldsymbol{\mathcal{D}}_{s_{2}+2N}$,  $ \cdots $, $ \boldsymbol{\mathcal{D}}_{s_{1}+s_{2}} $ can be defined similarly.

Note that the difference between  $ \| \boldsymbol{x}^{*}\left(\boldsymbol{\mathcal{D}}_{s_{1}}\right)-\boldsymbol{x}^{*}\left(\boldsymbol{\mathcal{D}}_{s_{2}}\right) \|$ can be decomposed into a series of similar structures such as $ \left\|\boldsymbol{x}^{*}\!\left(\boldsymbol{\mathcal{D}}_{s_{1}}\right)\!-\!\boldsymbol{x}^{*}\!\left(\boldsymbol{\mathcal{D}}_{s_{1}+N}\right)\!\right\| $, $ \left\|\boldsymbol{x}^{*}\!\left(\boldsymbol{\mathcal{D}}_{s_{1}+N}\right)\!-\!\boldsymbol{x}^{*}\!\left(\boldsymbol{\mathcal{D}}_{s_{1}+2N}\right)\!\right\| $, and so on. Hence, we only need to investigate  $ \left\|\boldsymbol{x}^{*}\left(\boldsymbol{\mathcal{D}}_{s_{1}}\right)\right.$ $\left.-\boldsymbol{x}^{*}\left(\boldsymbol{\mathcal{D}}_{s_{1}+N}\right)\right\| $. 

Assume that   $  H(\mathcal{D}_{s_{1,i}}^{i},\mathcal{D}_{s_{2,i}}^{i}) \leq \eta $ for  $  i \in \mathcal{I}$ and a positive constant $ \eta$. 	Due to the Hausdorff metric on convex and compact sets, there holds
$$  H(\mathcal{D}_{s_{1,i}}^{i},\mathcal{D}_{s_{1,i}+1}^{i})\leq H(\mathcal{D}_{s_{1,i}}^{i},\mathcal{D}_{s_{2,i}}^{i}) \leq \eta. $$		
By  Lemma \ref{l1},   when
$ H(\mathcal{D}_{s_{1,i}}^{i},\mathcal{D}_{s_{1,i}+1}^{i})\leq \eta$, the $ l $th row of $ B^{1,i} $ and $ B^{(1,{i})+1} $ satisfy
$$
\begin{aligned}
	&\left\|B_{l}^{1, i}-B_{j(l)}^{(1,{i})+1}\right\| \rightarrow0,\quad\mathrm{as}\;\eta\rightarrow0.\\
\end{aligned}
$$	
Correspondingly, $ \left\|b^{1,i}-b^{(1,{i})+1}\right\|\rightarrow0 $ as  $  \eta\rightarrow0 $. Then    $ \boldsymbol{\mathcal{D}}_{s_{1}}\rightarrow \boldsymbol{\mathcal{D}}_{s_{1}+N}$ as  $  \eta\rightarrow0 $. Since $ \operatorname{SOL}(\boldsymbol{\mathcal{D}}_{s}, F(\boldsymbol{x}))$ exists as an isolated solution, by \cite[Proposition 5.4.1]{facchinei2007finite},
$$
\begin{aligned}
	&\operatorname{SOL}\left(\boldsymbol{\mathcal{D}}_{s_{1}}, F(\boldsymbol{x})\right)\rightarrow\operatorname{SOL}\left(\boldsymbol{\mathcal{D}}_{s_{1}+N}, F(\boldsymbol{x})\right) , \quad\mathrm{as}\;\eta\rightarrow0.
\end{aligned}
$$
Therefore, for any $ \epsilon > 0 $, there exists $ \eta > 0 $ such that if  $ H(\mathcal{D}_{s_{1,i}}^{i},\mathcal{D}_{s_{1,i}+1}^{i})<\eta $, then
$$
\begin{aligned}
	&\left\|\boldsymbol{x}^{*}\left(\boldsymbol{\mathcal{D}}_{s_{1}}\right)-\boldsymbol{x}^{*}\left(\boldsymbol{\mathcal{D}}_{s_{1}+N}\right)\right\| \\
	= &\left\|\operatorname{SOL}\left(\boldsymbol{\mathcal{D}}_{s_{1}}, F(\boldsymbol{x})\right)-\operatorname{SOL}\left(\boldsymbol{\mathcal{D}}_{s_{1}+N}, F(\boldsymbol{x})\right)\right\| \\
	\leq & \epsilon.
\end{aligned}
$$		
Similarly, 
$$
\begin{aligned}		
	&\left\|\boldsymbol{x}^{*}\left(\boldsymbol{\mathcal{D}}_{s_{1}}\right)-\boldsymbol{x}^{*}\left(\boldsymbol{\mathcal{D}}_{s_{2}}\right)\right\| \\
	\leq  &  \left\|\boldsymbol{x}^{*}\left(\boldsymbol{\mathcal{D}}_{s_{1}}\right)-\boldsymbol{x}^{*}\left(\boldsymbol{\mathcal{D}}_{s_{1}+N}\right)\right\|\\
	& +\cdots+\left\|\boldsymbol{x}^{*}\left(\boldsymbol{\mathcal{D}}_{s_{1}+s_{2}-N}\right)-\boldsymbol{x}^{*}\left(\boldsymbol{\mathcal{D}}_{s_{1}+s_{2}}\right)\right\|\\
	&+\left\|\boldsymbol{x}^{*}\left(\boldsymbol{\mathcal{D}}_{s_{2}}\right)-\boldsymbol{x}^{*}\left(\boldsymbol{\mathcal{D}}_{s_{2}+N}\right)\right\| \\
	&+\cdots+\left\|\boldsymbol{x}^{*}\left(\boldsymbol{\mathcal{D}}_{s_{1}+s_{2}-N}\right)-\boldsymbol{x}^{*}\left(\boldsymbol{\mathcal{D}}_{s_{1}+s_{2}}\right)\right\|\\
	\leq  & s\epsilon,		
\end{aligned}
$$		
which means that $  \boldsymbol{x}^{*}(\boldsymbol{\mathcal{D}}_{s}) $ is continuous in $ \boldsymbol{\mathcal{D}}_{s} $ under the
Hausdorff metric. Moreover,   by
$\lim _{s \rightarrow \infty} H\left(\boldsymbol{\mathcal{D}}_{s_{1}},\boldsymbol{\Omega}\right) =0$ in  Lemma \ref{l2}, we have
$$
\lim \limits_{s \rightarrow \infty} \boldsymbol{x}^{*}\left(\boldsymbol{\mathcal{D}}_{s}\right)=\boldsymbol{x}^{*}(\boldsymbol{\Omega}),
$$	
where $ \boldsymbol{x}^{*}(\boldsymbol{\Omega}) $ is the Nash equilibrium of game (\ref{f1}).

Finally, we analyze the difference  between $  J_{i}(\boldsymbol{x}^{*}(\boldsymbol{\mathcal{D}}_{s}))$ and $J_{i}(x_{i}^{\prime},\boldsymbol{x}^{*}_{-i}(\boldsymbol{\mathcal{D}}_{s})) $,  where the $ i $th player’s equilibrium strategy is $ x^{*}_{i}(\boldsymbol{\mathcal{D}}_{s}) $   with  respect to $ \boldsymbol{\mathcal{D}}_{s} $ and $ x^{\prime}_{i} $ is arbitrarily chosen from $ \Omega_{i} $, while other players’ strategies remain the same $ \boldsymbol{x}^{*}_{-i}(\boldsymbol{\mathcal{D}}_{s}) $.
When  $H\left(D_{s_{i}}^{i}, \Omega_{i} \right)  \leq \eta$ for $ i\in \mathcal{I} $,
$$
\begin{aligned}
	&J_{i}\left(\boldsymbol{x}^{*}\left(\boldsymbol{\mathcal{D}}_{s}\right)\right)-J_{i}\left(x_{i}^{\prime}, \boldsymbol{x}_{-i}^{*}\left(\boldsymbol{\mathcal{D}}_{s}\right)\right) \\
	\leq &\left\|J_{i}\left(x_{i}^{\prime}, \boldsymbol{x}_{-i}^{*}(\boldsymbol{\Omega})\right)-J_{i}\left(x_{i}^{\prime}, \boldsymbol{x}_{-i}^{*}\left(\boldsymbol{\mathcal{D}}_{s}\right)\right)\right\|\\
	&+\left\|J_{i}\left(\boldsymbol{x}^{*}\left(\boldsymbol{\mathcal{D}}_{s}\right)\right)-J_{i}\left(\boldsymbol{x}^{*}(\boldsymbol{\Omega})\right)\right\|\\ &+J_{i}\left(\boldsymbol{x}^{*}(\boldsymbol{\Omega})\right)-J_{i}\left(x_{i}^{\prime}, \boldsymbol{x}_{-i}^{*}(\boldsymbol{\Omega})\right). \\
	\leq & \varsigma_{i}\left\|\boldsymbol{x}^{*}\left(\boldsymbol{\mathcal{D}}_{s}\right)-\boldsymbol{x}^{*}(\boldsymbol{\Omega})\right\|+\varsigma_{i}\left\|\boldsymbol{x}_{-i}^{*}(\boldsymbol{\Omega})-\boldsymbol{x}_{-i}^{*}\left(\boldsymbol{\mathcal{D}}_{s}\right)\right\|+0 \\
	\leq & 2\varsigma_{i}\epsilon,
\end{aligned}
$$		
where  $ \varsigma_{i} $ is the Lipschitz constant of $ J_{i} $. This completes the proof.
\end{proof}

Theorem \ref{t1},  based on convex set geometry and metric spaces, transforms the considered game into a variational problem. The accuracy of  $ \epsilon $-NE is influenced by several factors, specifically, the   vertices number of the approximate inscribed polyhedrons, the Lipschitz constants of payoff functions $ J_{i}(\boldsymbol{x}) $ for $ i \in \mathcal{I} $, and geometric structures of convex sets $ \Omega_{i} $ (referring to the constant $ C_{\Omega} $ with $\Omega= \Omega_{i} $ in Lemma \ref{l2}).
Obviously, when  constructing polyhedrons with more vertices, {we obtain more  hyperplanes  enclosed  the polyhedrons} (more rows of matrix $ B^{i} $ and vectors $ b^{i} $), which results in lower $ H\left(\mathcal{D}_{s_{i}}^{i}, \Omega_{i} \right) $ (referring to $ \mathcal{D}_{s}=\mathcal{D}_{s_{i}}^{i} $ and  $\Omega= \Omega_{i} $ in Lemma \ref{l2}) and   higher accuracy of $ \epsilon $. This conforms with the  intuition.

Actually,  there have been  methods on how to construct a proper inscribed polyhedron
such that
its vertices or faces are approximate to the convex set in
the best way. In other words, the Hausdorff
metric between  the convex set and the inscribed polyhedron can satisfy Lemma \ref{l2}.  Briefly, we introduce some methods for constructing an approximation polyhedron.

When the vertices or faces are constructed successively,
we can design  iterative algorithms  to find the best inscribed polyhedron. 
The main idea of iterative algorithms is to construct a polyhedron $\mathcal{D}^{k+1}=\operatorname{conv}\left(\mathcal{D}^{k} \cup\left\{w_{k+1}\right\}\right)$ every iteration, where $ w_{k+1} $ is a point   from $\partial \Omega $
(i.e., the boundary of $ \Omega $).	

One of the methods of constructing point $ w_{k+1} $
is described as follows.	For $ u\in \mathbb{R}^{n} $, denote $g_{\Omega}(u)=\max \{\langle u, x\rangle: x \in \Omega\}$ as the support function of $ \Omega $ on the unit sphere of directions $S^{n-1}=\left\{u \in \mathbb{R}^{n}:\|u\|=1\right\}$. 	The additional point $  w_{k+1} \in\partial \Omega $
belongs to the support plane parallel to the hyperplane in $  \mathcal{D}^{k} $, for which the quantity $ g_{\Omega}(u)- g_{\mathcal{D}^{k}}(u)$  attains its maximum on the set of external normals $  u \in S^{n-1} $ to the hyperplanes of $ \mathcal{D}^{k} $ \cite{2008Approximation}. Meanwhile, the initial polyhedron could be constructed by  the method \cite{bushenkov1985iteration}.	

Additionally,  the efficiency of the algorithm in the class of ellipsoids was described in \cite{dzholdybaeva1992numerical}.
	For sets with twice differentiable boundaries and positive curvatures,   the improved  approximation algorithms were proposed in \cite{2008Approximation,dzholdybaeva1992numerical}. For sets with nonsmooth boundaries, the  convergence velocity of algorithms was obtained in  \cite{Kamenev1993The}. 
Since the set constraint of each player is private information to itself, different players can approximate their feasible sets through different construction methods separately, in advance and offline. 
Therefore,  the computational cost and complexity of constructing vertices or faces of inscribed polyhedrons do not affect the  computational efficiency of the distributed algorithm essentially.
%
%

	\section{Distributed algorithm}\label{s4}
In this section, we  propose a distributed   algorithm for  the approximate game (\ref{e9})  and  investigate its  convergence performance.

In fact, each player has its own choices for  approximation,  with  local objective function $ J_{i}(x_{i},\boldsymbol{x}_{-i}) $, local approximate  set constraints $ \mathcal{D}_{s_{i}}^{i} $, $ B^{i} $, and $ b^{i} $ is  private knowledge of player $ i $. In  multi-agent frameworks, it is considered that player $ i $ can communicate with its neighbors through a network. Then 
we propose Algorithm 1 for seeking the $ \epsilon $-NE.

Let $ \beta_{1} $, $ \beta_{2} >0$ be some constants satisfying
\begin{subequations}\label{12}
	\begin{align}
		&0<\beta_{1}<\frac{2 \kappa}{c^{2}},\\
		&\beta_{2}>\frac{2 c_{2} \cdot c_{3}(2+\beta_{1} \cdot \kappa+2 \beta_{1} \cdot c)}{\lambda\left(2 \kappa-\beta_{1} \cdot c^{2}\right)} ,
	\end{align}
\end{subequations}
where $c \triangleq c_{1}+c_{2}\cdot c_{3}$, and $
\lambda $ is the smallest positive eigenvalue of $ \frac{1}{2}(L + L^{\mathrm{T}})  $ ($ L $ is the Laplacian matrix). Actually, the information of the eigenvalue $
\lambda $ can be obtained  by a distributed
method given in  \cite{cherukuri2016initialization} in advance. Thus, under Assumption 1, the  value of $
\lambda $, the strongly monotone parameter $\kappa  $,   and Lipschitz constants guarantee that the appropriate values of $ \beta_{1} $ and
$ \beta_{2} $ can always be obtained.

Since the $ i $th player's local feasible set $ \Omega_{i} $ is approximated by  inscribed polyhedron $ \mathcal{D}_{s_{i}}^{i} $ offline, the algorithm  contains a subproblem for solving a standard quadratic programming problem $ \operatorname{QP}(x_{i},\zeta_{i}) $ at each step  \cite{2010The},  defined as
\begin{equation}\label{q1}
	\min _{y}\left\|\left(x_{i}-\beta_{1} U_{i}\left(x_{i}, \zeta_{i}\right)\right)-y\right\|^{2}, \text { s.t. } B^{i} y \leq b^{i},
\end{equation}
where $U_{i}  $ was defined in (\ref{aaa}), $ B^{i} y \leq b^{i} $ is  equivalent to $ y\in \mathcal{D}_{s_{i}}^{i}$ in (\ref{bbb}). Denote $ \operatorname{SOL-QP}(x_{i},\zeta_{i}) $ as the solution to the  QP problem (\ref{q1}).  Thus, the distributed approximate  algorithm to solve  game (\ref{e9}) is designed as follows.

\begin{algorithm}[H]
	\renewcommand{\thealgorithm}{1}
	\caption{for each  $ i \in \mathcal{I}  $}
	\label{a1}
	\vspace{0.1cm}
	
	\textbf{Initialization}:
	\vspace{-0.2cm}
	\begin{flalign*}
		& x_{i}(0),y_{i}(0)\in \mathcal{D}_{s_{i}}^{i},\; \phi_{i}(0)=\boldsymbol{0}_{M},\;\zeta_{i}(0)=q_{i}(x_{i}(0)).
		&
	\end{flalign*}
	\vspace{-0.6cm}
	
	\textbf{Dynamics renewal}:
	\vspace{-0.2cm}
	\begin{flalign*}
		&\dot{x}_{i}=y_{i}-x_{i}, \\
		&\dot{\phi_{i}}=\beta_{2}\sum_{j =1}^{N}a_{i j}\left(\zeta_{j}-\zeta_{i}\right),\\
		&\zeta_{i}=\phi_{i}+q_{i}\left(x_{i}\right),\\
		&	y_{i}=\operatorname{SOL-QP}(x_{i},\zeta_{i}),
		&	 	
	\end{flalign*}
	where $ a_{ij} $ is the $ (i,j) $th element of the adjacency matrix.
\end{algorithm}
In  Algorithm 1, the $ i $th player calculates the local decision variable $ x_{i} \in \mathcal{D}_{s_{i}}^{i}$
based on projected gradient play dynamics by solving a $ \operatorname{QP}(x_{i},\zeta_{i}) $  problem at each step. The local variable $ \zeta_{i}$  is to estimate  the global aggregation $ \mathcal{Q}(\boldsymbol{x}) $. 
The design idea is improved based on \cite{ye2016game,liang2019distributed}, in which the projection in our algorithm is obtained with quadratic programming, thus improving the computational efficiency.
\begin{remark}\label{444}
	Quadratic programming in Algorithm 1 ensures that the projection is solvable in polynomial time, even with a large number of linear inequality constraints, while the general nonlinear programming corresponding to the high-dimensional nonlinear constraints cannot guarantee this \cite{vavasis1991nonlinear}.  		
	For example,
	the computational  cost of the projection on ellipsoid constraints is $ O (n^4) $ \cite{ye2014new}, whereas it is $ O (n^{2.5}) $ on linear constraints caused by approximation \cite{monteiro1989interior}, especially $ O (n) $ if linear constraints are generalized bounded constraints \cite{calamai1987projected}. More details about the computational cost saved by approximation  are explained  by numerical experiments in Section \ref{s6}.
\end{remark}

A compact form of Algorithm 1 can be written as
\begin{equation}\label{e10}
	\left\{\begin{array}{ll}
		\dot{\boldsymbol{x}}=\boldsymbol{y}-\boldsymbol{x}, & \boldsymbol{x}(0) \in \boldsymbol{\mathcal{D}}_{s}, \\
		\dot{\boldsymbol{\zeta}}=-\beta_{2} L \otimes I_{M} \boldsymbol{\zeta}+\frac{d}{d t} \boldsymbol{q}(\boldsymbol{x}), & \boldsymbol{\zeta}(0)=\boldsymbol{q}(\boldsymbol{x}(0)),
	\end{array}\right.
\end{equation}
where    $ \boldsymbol{\zeta}= col(\zeta_{1},\!\cdots\!, \zeta_{N})$, $ \boldsymbol{q}(\boldsymbol{x})\!=\! col(q_{1}(x_{1}\!),\!\cdots\!, q_{N}(x_{N})\!)$, $ \boldsymbol{y}= col(y_{1},\cdots, y_{N})$ with $ y_{i}=\operatorname{SOL-QP}(x_{i},\zeta_{i}) $ basically.

Then we first verify the equivalency between the equilibrium of dynamics (\ref{e10}) and the
Nash equilibrium  $ 	\boldsymbol{x}^{*}\left(\boldsymbol{\mathcal{D}}_{s}\right)  $   of (\ref{e9}), whose   proof  is straightforward by Lemma 1 and Lemma 2.

\begin{lemma}\label{l10}
	Under Assumption 1, the  equilibrium of  (\ref{e10})  is
	\begin{equation}\label{e27}
		\left[\begin{array}{l}
			\boldsymbol{x} \\
			\boldsymbol{\zeta}
		\end{array}\right]=\left[\begin{array}{l}
			\boldsymbol{x}^{*}\left(\boldsymbol{\mathcal{D}}_{s}\right) \\
			\boldsymbol{\zeta}^{*}\left(\boldsymbol{\mathcal{D}}_{s}\right)
		\end{array}\right]=\left[\begin{array}{c}
			\boldsymbol{x}^{*}\left(\boldsymbol{\mathcal{D}}_{s}\right) \\
			\boldsymbol{1}_{N} \otimes \mathcal{Q}\left(\boldsymbol{x}^{*}\left(\boldsymbol{\mathcal{D}}_{s}\right)\right)
		\end{array}\right],
	\end{equation}
	where $ 	\boldsymbol{x}^{*}\left(\boldsymbol{\mathcal{D}}_{s}\right)  $ is  the NE of approximate game (\ref{e9}).
\end{lemma}

From Lemma 2,  the strong connectivity and weight balance of graph 	
$ \mathcal{G} $  guarantee $\zeta_{1}=\zeta_{2}=\cdots=\zeta_{N}$, and $ \frac{1}{N}\sum_{i=1}^{N}\zeta_{i}=\mathcal{Q}(\boldsymbol{x}) $.  Together with Lemma 1, the point given in (\ref{e27}) is
the equilibrium of (\ref{e10}). Moreover, by Lemma \ref{l10}, the convergence of Algorithm 1 is discussed in the following lemma, by easily extending \cite[Theorem 2]{liang2019distributed}.
%
\begin{lemma}\label{l11}
	Under Assumption 1, the  algorithm (\ref{e10})  converges  at an exponential  rate. Moreover, $ \boldsymbol{x} $ in (\ref{e10}) exponentially converges to
	the NE of  (\ref{e9}).
\end{lemma}

Furthermore, from Lemma  \ref{l10},
take
$$
\boldsymbol{\sigma}\triangleq\boldsymbol{\zeta} -\mathbf{1}_{N} \otimes \mathcal{Q}(\boldsymbol{x}).
$$
The distributed
algorithm (\ref{e10})  of the approximate game (\ref{e9}) can be written via a general distributed projected gradient dynamics, by  $L \mathbf{1}_{N}=\mathbf{0}_{N}$ and (\ref{e10}), as follows:
\begin{subequations}\label{e17}
	\begin{align}
		\dot{\boldsymbol{x}}&=\Pi_{\boldsymbol{\mathcal{D}}_{s}}(\boldsymbol{x}-\beta_{1} U(\boldsymbol{x},\mathbf{1}_{N} \otimes \mathcal{Q}(\boldsymbol{x})+\boldsymbol{\sigma} ))-\boldsymbol{x},\\
		\dot{\boldsymbol{\sigma}}&=-\beta_{2} L \otimes I_{M} \boldsymbol{\sigma}+\frac{d}{d t}\left(\boldsymbol{q}(\boldsymbol{x})-\mathbf{1}_{N} \otimes \mathcal{Q}(\boldsymbol{x})\right) \\\notag
		&=-\beta_{2} L \otimes I_{M} \boldsymbol{\sigma}+\left(\nabla \boldsymbol{q}(\boldsymbol{x})-\mathbf{1}_{N} \otimes \nabla \mathcal{Q}(\boldsymbol{x})\right)^{\mathrm{T}}\cdot \\\notag
		&\quad\;\left(\Pi_{\boldsymbol{\mathcal{D}}_{s}}\left(\boldsymbol{x}-\beta_{1} U\left(\boldsymbol{x}, \mathbf{1}_{N} \otimes \mathcal{Q}(\boldsymbol{x})+\boldsymbol{\sigma}\right)\right)-\boldsymbol{x}\right),\notag
	\end{align}
\end{subequations}
where $ \boldsymbol{x}(0) \in \boldsymbol{\mathcal{D}}_{s} $ and $ \boldsymbol{\sigma}(0)=\boldsymbol{q}(\boldsymbol{x}(0))-\mathbf{1}_{N} \otimes \mathcal{Q}(\boldsymbol{x}(0)) $.

Analogously, the distributed algorithm for original game (\ref{f1}) (without any approximation) can be written as
\begin{equation}\label{11}
	\begin{array}{l}
		\dot{\boldsymbol{x}}=\Pi_{\boldsymbol{\Omega}}(\boldsymbol{x}-\beta_{1} U(\boldsymbol{x}, \mathbf{1}_{N} \otimes \mathcal{Q}(\boldsymbol{x})+\boldsymbol{\sigma}))-\boldsymbol{x},\\
		\dot{\boldsymbol{\sigma}}=-\beta_{2} L \otimes I_{M} \boldsymbol{\sigma}+\left(\nabla \boldsymbol{q}(\boldsymbol{x})-\mathbf{1}_{N} \otimes \nabla \mathcal{Q}(\boldsymbol{x})\right)^{\mathrm{T}}\cdot \\
		\quad\quad\left(\Pi_{\boldsymbol{\Omega}}\left(\boldsymbol{x}-\beta_{1} U\left(\boldsymbol{x}, \mathbf{1}_{N} \otimes \mathcal{Q}(\boldsymbol{x})+\boldsymbol{\sigma}\right)\right)-\boldsymbol{x}\right),
	\end{array}
\end{equation}
where $ \boldsymbol{x}(0) \in \boldsymbol{\Omega} $ and $ \boldsymbol{\sigma}(0)=\boldsymbol{q}(\boldsymbol{x}(0))-\mathbf{1}_{N} \otimes \mathcal{Q}(\boldsymbol{x}(0))$.

For clarification, let $\boldsymbol{m}\triangleq-\beta_{2} L \otimes I_{M} \boldsymbol{\sigma}$, $ \, $ $ \boldsymbol{\rho}\triangleq\nabla \boldsymbol{q}(\boldsymbol{x})-\mathbf{1}_{N} \otimes \nabla \mathcal{Q}(\boldsymbol{x}) $ in (\ref{e17}) and (\ref{11}). Denote  $\boldsymbol{z}=col\left\{\boldsymbol{x},\boldsymbol{\sigma}\right\} \in \mathbb{R}^{nN+MN}$. Then a compact form of   (\ref{11}) can be written as
\begin{equation}\label{e11}	
	\dot{\boldsymbol{z}}=G_{\boldsymbol{\Omega}}(\boldsymbol{z}),
\end{equation}
where $$
G_{\boldsymbol{\Omega}}(\boldsymbol{z})\!=\!\left[\begin{array}{l}\!
	\Pi_{\boldsymbol{\Omega}}(\boldsymbol{x}-\beta_{1} U(\boldsymbol{x},  \mathbf{1}_{N} \otimes \mathcal{Q}(\boldsymbol{x})+\boldsymbol{\sigma}))\!-\!\boldsymbol{x}\\
	\boldsymbol{m}\!+\!\boldsymbol{\rho}^{\mathrm{T}}\!\! \left(\Pi_{\boldsymbol{\Omega}}\!\left(\boldsymbol{x}\!-\!\beta_{1} U\!\left(\boldsymbol{x}, \mathbf{1}_{N} \!\otimes\! \mathcal{Q}(\boldsymbol{x})\!+\!\boldsymbol{\sigma}\right)\right)\!-\!\boldsymbol{x}\right)
\end{array}\!\right].
$$

		In essence, from Lemma \ref{l11}, the conclusion of exponential convergence  is also applicable to (\ref{e11}). According to this property,
		it follows from the converse theorem for exponentially stable systems  \cite[Theorem 4.14]{khalil2002nonlinear} that there exists a  Lyapunov  function $ V_{\boldsymbol{\Omega}}(\boldsymbol{z}) $ of  (\ref{e11}) satisfying the following inequalities,
\begin{equation}\label{e22}	
	\begin{array}{c}
		a_{1}\|\boldsymbol{z}-\boldsymbol{z}^{*}(\boldsymbol{\Omega})\|^{2} \leq V_{\boldsymbol{\Omega}}(\boldsymbol{z}) \leq a_{2}\|\boldsymbol{z}-\boldsymbol{z}^{*}(\boldsymbol{\Omega})\|^{2}, \\
		\dot{V}_{\boldsymbol{\Omega}} \leq-a_{3}\|\boldsymbol{z}-\boldsymbol{z}^{*}(\boldsymbol{\Omega})\|^{2}, \\
		\left\|\frac{\partial V_{\boldsymbol{\Omega}}}{\partial \boldsymbol{z}}\right\| \leq a_{4}\|\boldsymbol{z}-\boldsymbol{z}^{*}(\boldsymbol{\Omega})\|,
	\end{array}
\end{equation}
where $a_{1}$, $a_{2}$, $a_{3}$, and $a_{4}$ are positive constants, and $ \boldsymbol{z}^{*}(\boldsymbol{\Omega})=col\left\{\boldsymbol{x}^{*}(\boldsymbol{\Omega}),\boldsymbol{\sigma}^{*}\left(\boldsymbol{\Omega}\right) \right\}$ is the  exponentially stable equilibrium point of system (\ref{e11}).

Moreover,  (\ref{e17})  can be rewritten as
\begin{equation}\label{e60}	
	\dot{\boldsymbol{z}}=G_{\boldsymbol{\mathcal{D}}_{s}}(\boldsymbol{z}),
\end{equation}
with
$$
G_{\boldsymbol{\mathcal{D}}_{s}}(\boldsymbol{z})\!=\!\left[\begin{array}{l}\!
	\Pi_{\boldsymbol{\mathcal{D}}_{s}}(\boldsymbol{x}-\beta_{1} U(\boldsymbol{x}, \mathbf{1}_{N} \otimes \mathcal{Q}(\boldsymbol{x})+\boldsymbol{\sigma}))\!-\!\boldsymbol{x}\\
	\boldsymbol{m}\!+\!\boldsymbol{\rho}^{\mathrm{T}}\!\! \left(\Pi_{\boldsymbol{\mathcal{D}}_{s}}\!\left(\boldsymbol{x}\!-\!\beta_{1} U\!\left(\boldsymbol{x}, \mathbf{1}_{N} \!\otimes\! \mathcal{Q}(\boldsymbol{x})\!+\!\boldsymbol{\sigma}\right)\right)\!-\!\boldsymbol{x}\right)
\end{array}\!\right],
$$
which can be regarded as a  perturbed system of   (\ref{e11}).
Denote 	
$$ e(\boldsymbol{z})\triangleq G_{\boldsymbol{\mathcal{D}}_{s}}(\boldsymbol{z})-G_{\boldsymbol{\Omega}}(\boldsymbol{z}).$$Consequently,  system  (\ref{e60}) can be rewritten as
\begin{equation}\label{e12}	
	\dot{\boldsymbol{z}}
	=G_{\boldsymbol{\Omega}}(\boldsymbol{z})+e(\boldsymbol{z}),
\end{equation}
with the perturbation term as
$$
e(\boldsymbol{z})\!=\!\left[\!\begin{array}{l}
	\Pi_{\boldsymbol{\mathcal{D}}_{s}}(\boldsymbol{x}-\beta_{1} U(\boldsymbol{x}, \boldsymbol{\zeta}))\!-\!\Pi_{\boldsymbol{\Omega}}(\boldsymbol{x}-\beta_{1} U(\boldsymbol{x}, \boldsymbol{\zeta}))\\
	\boldsymbol{\rho}^{\mathrm{T}} \!\!
	\left(\Pi_{\boldsymbol{\mathcal{D}}_{s}}\left(\boldsymbol{x}-\beta_{1} U(\boldsymbol{x}, \boldsymbol{\zeta})\right)\!-\!\Pi_{\boldsymbol{\Omega}}\left(\boldsymbol{x}-\beta_{1} U(\boldsymbol{x}, \boldsymbol{\zeta})\right)\right)
\end{array}\!\!\right]\!.
$$

Then we investigate the upper bound of $ \epsilon $.
Note that  $ e(\boldsymbol{z}) $ reflects the difference   in projected dynamics on  inscribed polyhedrons $ \boldsymbol{\mathcal{D}}_{s} $ and  original action sets $ \boldsymbol{\Omega} $, respectively.  It is  essentially caused by  the  approximation of game (\ref{f1}).  Consider an arbitrary approximate construction  based on Hausdorff distances
$ \boldsymbol{H}=col(h_{1},\cdots,h_{N}) $, where $h_{i}=H(\Omega_{i},\mathcal{D}_{s_{i}}^{i})$ represents the Hausdorff  distance between the original set  $ \Omega_{i}$ and its inscribed polyhedron $ \mathcal{D}_{s_{i}}^{i} $ for $ i\in \mathcal{I}$.
Then the following lemma shows an  upper bound of   $ \|e(\boldsymbol{z})\| $, whose proof is given in Appendix \ref{a13}
\begin{lemma}\label{t3}	
	Under Assumption 1,   given the Hausdorff distances $\boldsymbol{H}$, we have
	\begin{align}\label{e24}	
		\|e(\boldsymbol{z})\|&\leq\delta(\boldsymbol{H})\\
		&=(1+c_{3})\sqrt{ \sum_{i=1}^{N}\left(\frac{2}{ \nu_{i}}\operatorname{arccos}(1-\nu_{i}h_{i})+h_{i}\right)^{2}},\notag
	\end{align}
	where $ c_{3} $ is the  Lipschitz constant of $ q_{i} $,  and $\nu_{i}$ is a constructive curvature related merely to the structure of $ \Omega_{i}$  for  $ i\in \mathcal{I} $.
\end{lemma}	

From Lemma \ref{t3}, since  $\nu_{i}$ is independent of any approximation  of $ \Omega_{i} $,
the bound of $ e(\boldsymbol{z}) $ is explicitly affected by  Hausdorff  distances $\boldsymbol{H}$. Clearly, a lower metric yields a lower bounds of $ e(\boldsymbol{z}) $. Furthermore, the next lemma investigates the influence of  $ e(\boldsymbol{z}) $ on   perturbed system (\ref{e12}), whose proof  is shown in Appendix \ref{a5}
\begin{lemma} \label{l7}
	Let $ V_{\boldsymbol{\Omega}}(\boldsymbol{z}) $ be a Lyapunov  function  satisfying (\ref{e22}) in $  \Xi=\{\boldsymbol{z}\in \boldsymbol{\Omega}\times\mathbb{R}^{MN}|\|\boldsymbol{z}-\boldsymbol{z}^{*}(\boldsymbol{\Omega})\|<r \}  $, which is a compact set. Suppose
	\begin{equation}\label{39}
		\| e(\boldsymbol{z})\|\leq\delta(\boldsymbol{H})<\frac{a_{3}}{a_{4}}\sqrt{\frac{a_{1}}{a_{2}}}\mu r,
	\end{equation}
	for all $\boldsymbol{z}\in \Xi $ and  a positive constant $ \mu<1  $. Then, for all $ \|\boldsymbol{z}(t_{0})-\boldsymbol{z}^{*}(\boldsymbol{\Omega})\|\leq \sqrt{a_{1}/a_{2}}r $, the solution $ \boldsymbol{z}(t)  $ of the perturbed system (\ref{e12}) satisfies
	\begin{equation}\label{qd}
		\|\boldsymbol{z}(t)-\boldsymbol{z}^{*}(\boldsymbol{\Omega})\| \leq \sqrt{\frac{a_{2}}{a_{1}}} e^{-\omega (t-t_{0})} \left\|\boldsymbol{z}(t_{0})-\boldsymbol{z}^{*}(\boldsymbol{\Omega})\right\|, \;
	\end{equation}
	for $ t_{0} \leq t<t_{0}+T $, and
	\begin{equation}\label{qe}
		\|\boldsymbol{z}(t)-\boldsymbol{z}^{*}(\boldsymbol{\Omega})\| \leq R,
	\end{equation}
	for $ t\geq t_{0}+T $,
	where $ T $ is a finite positive scalar,
	$$\omega=\frac{(1-\mu)a_{3}}{2a_{2}},\; R=\frac{a_{4}}{\mu a_{3}} \sqrt{\frac{a_{2}}{a_{1}}}\delta(\boldsymbol{H}). $$	
\end{lemma}

Lemma \ref{l7} explains that  if  $ e(\boldsymbol{z}) $ is small enough, then    $ \left\| \boldsymbol{z}(t)-\boldsymbol{z}^{*}(\boldsymbol{\Omega})  \right\| $ of (\ref{e12}) is ultimately bounded by a small bound, where $ \boldsymbol{z}^{*}(\boldsymbol{\Omega}) = col\left\{\boldsymbol{x}^{*}(\boldsymbol{\Omega}),\boldsymbol{\sigma}^{*}(\boldsymbol{\Omega}) \right\}$ is the exponentially stable equilibrium of the nominal system  (\ref{e11}). Moreover, since $a_{1}$, $a_{2}$, $a_{3}$, $a_{4}$ and $\mu$  are constants,
the global exponential convergence of (\ref{e11}) guarantees that for any $\delta(\boldsymbol{H})$ and $ \|\boldsymbol{z}(t_{0})-\boldsymbol{z}^{*}(\boldsymbol{\Omega})\|$, we can choose $ r $ large enough to satisfy (\ref{39}) and the initial condition.  Therefore,  by the exponential convergence of (\ref{e11}), we can analyze the accuracy of $\epsilon$  based on continuous-time dynamics and bounded stability of perturbed systems. Obviously,    
from (\ref{qe}),  $ \left\| \boldsymbol{z}^{*}(\boldsymbol{\mathcal{D}}_{s})-\boldsymbol{z}^{*}(\boldsymbol{\Omega})  \right\| $ and $ \left\| \boldsymbol{x}^{*}(\boldsymbol{\mathcal{D}}_{s})-\boldsymbol{x}^{*}(\boldsymbol{\Omega})  \right\| $  are bounded, where $ \boldsymbol{z}^{*}\left(\boldsymbol{\mathcal{D}}_{s}\right)= col\left\{\boldsymbol{x}^{*}(\boldsymbol{\mathcal{D}}_{s}),\boldsymbol{\sigma}^{*}(\boldsymbol{\mathcal{D}}_{s}) \right\} $ is the equilibrium of (\ref{e12}).
Recalling the definition of $ \epsilon $-NE,  this upper bound of $ \left\| \boldsymbol{x}^{*}(\boldsymbol{\mathcal{D}}_{s})-\boldsymbol{x}^{*}(\boldsymbol{\Omega})  \right\| $  can be  regarded as a discrepancy proportional to the upper bound  of  $ \epsilon $.

Together with  Lemma \ref{t3} and Lemma \ref{l7},  the conclusion about the approximation accuracy  is shown in the following theorem.
\begin{mythm}\label{t2}	
	Under Assumption 1, 
	\begin{equation}\label{e23}	
		\begin{aligned}
			\epsilon&\leq \frac{2a_{4}}{a_{3}} \sqrt{\frac{a_{2}}{a_{1}}}\frac{\varsigma_{i}}{\mu}	\delta(\boldsymbol{H}),
		\end{aligned}
	\end{equation}
	where   the  constant $ \mu \in(0,1)  $, $a_{1}$, $a_{2}$, $a_{3}$, and $a_{4}$ are positive  constants in (\ref{e22}), $ \varsigma_{i} $ is the Lipschitz constant of $ J_{i} $,  and $ \delta(\boldsymbol{H}) $ is 
	defined in (\ref{e24}).
\end{mythm}
\begin{proof}
Similar to the last part in the proof of Theorem 1, recalling the definition of $ \epsilon $-NE,  the  difference  between $  J_{i}(\boldsymbol{x}^{*}(\boldsymbol{\mathcal{D}}_{s}))$ and $J_{i}(x_{i}^{\prime},\boldsymbol{x}^{*}_{-i}(\boldsymbol{\mathcal{D}}_{s})) $ satisfies
$$
\begin{aligned}
	&J_{i}\left(\boldsymbol{x}^{*}\left(\boldsymbol{\mathcal{D}}_{s}\right)\right)-J_{i}\left(x_{i}^{\prime}, \boldsymbol{x}_{-i}^{*}\left(\boldsymbol{\mathcal{D}}_{s}\right)\right) \\
	\leq &\left\|J_{i}\left(x_{i}^{\prime}, \boldsymbol{x}_{-i}^{*}(\boldsymbol{\Omega})\right)-J_{i}\left(x_{i}^{\prime}, \boldsymbol{x}_{-i}^{*}\left(\boldsymbol{\mathcal{D}}_{s}\right)\right)\right\|\\
	&+\left\|J_{i}\left(\boldsymbol{x}^{*}\left(\boldsymbol{\mathcal{D}}_{s}\right)\right)-J_{i}\left(\boldsymbol{x}^{*}(\boldsymbol{\Omega})\right)\right\|\\ &+J_{i}\left(\boldsymbol{x}^{*}(\boldsymbol{\Omega})\right)-J_{i}\left(x_{i}^{\prime}, \boldsymbol{x}_{-i}^{*}(\boldsymbol{\Omega})\right)\\
	\leq & \varsigma_{i}\left\|\boldsymbol{x}^{*}\left(\boldsymbol{\mathcal{D}}_{s}\right)-\boldsymbol{x}^{*}(\boldsymbol{\Omega})\right\|+\varsigma_{i}\left\|\boldsymbol{x}_{-i}^{*}(\boldsymbol{\Omega})-\boldsymbol{x}_{-i}^{*}\left(\boldsymbol{\mathcal{D}}_{s}\right)\right\|,\\
\end{aligned}
$$
where the $ i $th player’s equilibrium strategy is $ x^{*}_{i}(\boldsymbol{\mathcal{D}}_{s}) $    with  respect to $ \boldsymbol{\mathcal{D}}_{s} $ and $ x^{\prime}_{i} $ is arbitrarily chosen from $ \Omega_{i} $. Meanwhile, other players’ strategies remain the same $ \boldsymbol{x}^{*}_{-i}(\boldsymbol{\mathcal{D}}_{s}) $.

Due to  Lemma \ref{l7}, with $ \boldsymbol{z}^{*}\left(\boldsymbol{\mathcal{D}}_{s}\right)= col\left\{\boldsymbol{x}^{*}(\boldsymbol{\mathcal{D}}_{s}),\boldsymbol{\sigma}^{*}(\boldsymbol{\mathcal{D}}_{s}) \right\} $ as the equilibrium of (\ref{e12}), it follows from (\ref{qe}) that $
\left\|\boldsymbol{z}^{*}\left(\boldsymbol{\mathcal{D}}_{s}\right)-\boldsymbol{z}^{*}(\boldsymbol{\Omega})\right\|\leq R
$. Then
$$
\left\|\boldsymbol{x}^{*}\left(\boldsymbol{\mathcal{D}}_{s}\right)-\boldsymbol{x}^{*}(\boldsymbol{\Omega})\right\|\leq R.
$$

As a result,
$$
\begin{aligned}
	&\varsigma_{i}\left\|\boldsymbol{x}^{*}\left(\boldsymbol{\mathcal{D}}_{s}\right)-\boldsymbol{x}^{*}(\boldsymbol{\Omega})\right\|+\varsigma_{i}\left\|\boldsymbol{x}_{-i}^{*}(\boldsymbol{\Omega})-\boldsymbol{x}_{-i}^{*}\left(\boldsymbol{\mathcal{D}}_{s}\right)\right\|\\
	&	\leq \frac{2a_{4}}{a_{3}} \sqrt{\frac{a_{2}}{a_{1}}}\frac{\varsigma_{i}}{\mu}\delta(\boldsymbol{H}),
\end{aligned}
$$
which completes the  proof.		
\end{proof}

	\begin{remark}
	From  (\ref{e23}), the upper bound of  $ \epsilon $  is proportional to the  bound of $ e(\boldsymbol{z}) $,  which indicates that arbitrarily small perturbations will not cause a significant deviation. Moreover, it can be regarded  as the robustness  of the nominal system with an exponentially stable equilibrium. 
	Thus, with the help of the analysis in Section \ref{s3}, we show   the accuracy of $ \epsilon $ based on  bounded stability of perturbed systems,  and give an estimation of the  upper bound.
\end{remark}

\section{Numerical experiments}\label{s6}
We examine
the computational efficiency and  approximation accuracy of Algorithm 1 on  Nash-Cournot games and demand response management models in the following two subsections.

\subsection{For  approximation accuracy}

To illustrate the convergence and approximation, we consider a classical Cournot game played by $ N = 4 $ competitive players over a network as in \cite{koshal2016distributed} and \cite{liang2019distributed}. For $ i \in \mathcal{I}= \{1, \cdots, N\} $, the action set $ \Omega_{i} $ is an elliptical region that
$$\Omega_{i} = \mathbf{E}_{4,3}(0,0)=\left\{x_{i} \in \mathbb{R}^{2}:\frac{x_{i1}^{2}}{4^{2}}+\frac{x_{i2}^{2}}{3^{2}}\leq1\right\}.
$$
The payoff function $ f_{i}\left(x_{i}, \mathcal{Q}(\boldsymbol{x})\right) $ is
\begin{equation}\label{51}
	f_{i}\left(x_{i}, \mathcal{Q}(\boldsymbol{x})\right)=x_{i}^{\mathrm{T}}(d_{i}(x_{i})-p(\mathcal{Q}(\boldsymbol{x}))),
\end{equation}
where $ d_{i}(x_{i})= 0.5(x_{i}+(13-i)\boldsymbol{1}_{2})$
and $p=N\boldsymbol{1}_{2}-0.01\mathcal{Q}(\boldsymbol{x}) $
with
$
\mathcal{Q}(\boldsymbol{x})=\frac{1}{N} \sum_{j=1}^{N} x_{j}.
$

Clearly, the game model satisfies Assumption 1 with constants
$ \kappa=1 $, $ c_{1}=1.0025 $, $ c_{2}=0.01 $, and $ c_{3}=1 $. We adopt the following ring graph as the  network $ \mathcal{G} $,
$$
1\rightarrow2\rightarrow3\rightarrow4\rightarrow1.
$$
To render condition (\ref{12}),  assign $ \beta_{1} = 0.1 $ and $ \beta_{2} = 1 $. Also, set tolerance $ t_{tol}= 10^{-3} $ and the terminal criterions
$$ \|\dot{\boldsymbol{x}}(t)\| \leq t_{tol},\quad \|\dot{\boldsymbol{\zeta}}(t)\| \leq t_{tol},  $$
where $ \dot{\boldsymbol{x}}(t) $ and $ \dot{\boldsymbol{\zeta}}(t) $ were given in (\ref{e10}).

We present trajectories by approximating $ \mathbf{E}_{4,3}(0,0) $ with inscribed octagons.  The trajectories of one dimension of each strategy $ x_{i} $ are shown in Fig. \ref{fig4}. The strategies of all players converge to their corresponding equilibrium points with an  exponential rate, which verifies the correctness of our algorithm.
	\begin{figure}
	\centering	
	\includegraphics[scale=0.55]{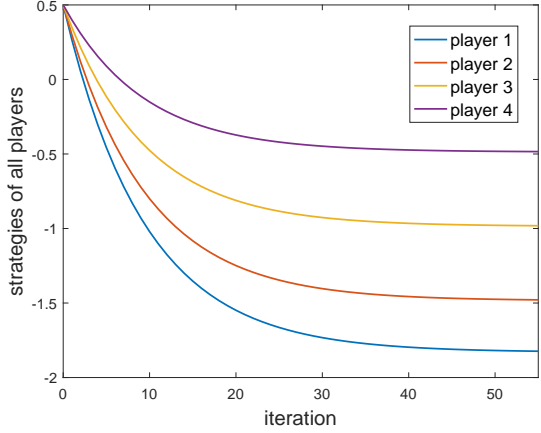}\\
	\caption{ Trajectories of all players’ strategies.}
	\label{fig4}
\end{figure}
\begin{figure}
	\centering	
	\includegraphics[scale=0.55]{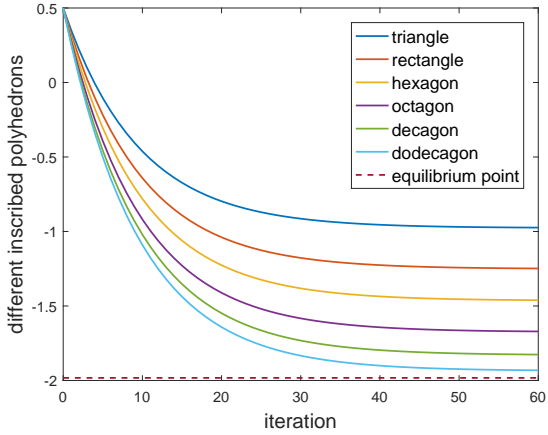}\\
	\caption{ Trajectories of approximation by different inscribed polyhedrons.}
	\label{fig6}
\end{figure}
\begin{table*}
	\centering
	\label{tab1}
	\normalsize
	\caption{Performance of different approximations.}
	\setlength\tabcolsep{12pt}
	\renewcommand\arraystretch{1.2}
	\begin{tabular}{c|c|c|c|c|c|c}
		\hline
		\hline
		Polyhedrons &  Triangle &  Rectangle &  Hexagon &  Octagon&Decagonal&Dodecagonal \\ \hline
		Values of $ \epsilon $ & 1.3470 & 0.8491 & 0.5187 & 0.2261 &0.1069&0.0473\\ \hline\hline
	\end{tabular}
\end{table*}
Fig. \ref{fig6} shows different strategy trajectories of one fixed player with inscribed triangles, rectangles, hexagons, octagons, decagons, and dodecagons to approximate $ \mathbf{E}_{4,3}(0,0) $, respectively. The vertical axis represents the value of the convergent $ \epsilon $-NE and the horizontal axis represents the iteration time of Algorithm 1. As can be seen from Fig. \ref{fig6},  equilibria with different polyhedrons get closer to the exact solution with more accurate approximations.

Moreover,  the numerical values of $ \epsilon $  under different types of approximation are listed in Table 1.
Obviously,  the value of $ \epsilon $ decreases with the increase of the edges of polyhedrons and the decrease of Hausdorff distances, which is consistent with the approximation results in the previous sections.

\subsection{For computational efficiency}
Here, we show the computational efficiency of Algorithm 1 on a class of demand response management problems under  various network scales and parameter settings.

Consider $ N $ electricity users with the demand of energy consumption as in \cite{ye2016game},  \cite{liang2017distributed}. For $ i \in \mathcal{I}= \{1, ..., N\} $, the action set $ \Omega_{i} $ is the energy consumption of the $ i $th user and $ f_{i}\left(x_{i}, \mathcal{Q}(\boldsymbol{x})\right) $ is the cost function in the following
form,
\begin{equation}\label{54}
	f_{i}\left(x_{i}, \mathcal{Q}(\boldsymbol{x})\right)=\iota_{i}(x_{i}-\pi_{i})^{\mathrm{T}}(x_{i}-\pi_{i})+ x_{i}^{\mathrm{T}}P (\mathcal{Q}(\boldsymbol{x})),
\end{equation}
where $ \iota_{i} $ is constant and $ \pi_{i} $ is the nominal value of energy consumption for $ i = \{1, ..., N\} $, with  $ P (\mathcal{Q}(\boldsymbol{x}))=\omega_{i}N\mathcal{Q}(\boldsymbol{x})+p_{0} $ and
\begin{equation}\label{52}
	\mathcal{Q}(\boldsymbol{x})=\frac{1}{N} \sum_{j=1}^{N} x_{j}.
\end{equation}
Set $ N=10 $, $ \iota_{i}=0.05 $, $ \omega_{i}=0.001 $, $ p_{0}=\boldsymbol{1}_{3} $,   and $ \pi_{i}=0.5(10-i)\boldsymbol{1}_{3}\in\mathbb{R}^{3}  $. Then the action set $ \Omega_{i} $ of each player is an elliptical region denoted by
$ \mathbf{E}_{7,6,5}(0,0,0) $.

Take a ring graph as the communication network $ \mathcal{G} $,
$$
1\rightarrow2\rightarrow\cdots\rightarrow10\rightarrow1,
$$
and assign $ \beta_{1} = 0.5 $ and $ \beta_{2} = 2 $ to meet the condition (\ref{12}). Besides, we set tolerance $ t_{tol}= 10^{-3} $.

\begin{figure}
	\centering	
	\includegraphics[scale=0.55]{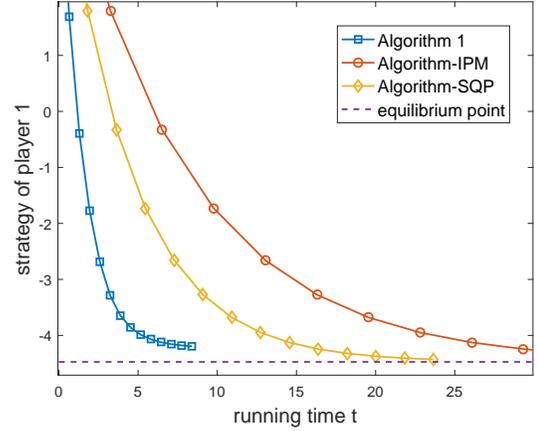}\\
	\caption{Strategy trajectories of player 1   with different algorithms.}
	\label{fig7}
\end{figure}

Here, we  use numerical optimization to directly process the  projections on nonlinear constraints $ \mathbf{E}_{7,6,5}(0,0,0) $ for comparison. Fig. \ref{fig7}  shows the different strategy trajectories of one fixed player in dynamics (\ref{11})  by  Algorithm 1  (i.e., with approximation), the algorithm  based on sequential quadratic program (Algorithm-SQP) and the algorithm  based on the interior point method (Algorithm-IPM) (i.e., without approximation). Algorithm 1 makes projections  on the  inscribed polyhedrons of $ \mathbf{E}_{7,6,5}(0,0,0) $ with the number of vertice $s=12$, while Algorithm-SQP and Algorithm-IPM  make projections on  $ \mathbf{E}_{7,6,5}(0,0,0) $ directly. In Fig. \ref{fig7}, the vertical axis represents the value of the convergent equilibria, and the horizontal axis represents the real running time in seconds. Clearly, Algorithm 1 converges faster, although it does not converge to the exact  equilibrium point. However, from the error shown in Fig. \ref{fig7}, this sacrifice is tolerable.

Moreover, according to Remark \ref{444}, the  complexity of Algorithm 1 can be roughly characterized as $ O(Nn^{2.5}) $, while is $ O(Nn^{4}) $ for Algorithm-SQP and Algorithm-IPM on ellipsoid constraints. To further illustrate the computational cost saved by approximation,
we report the performance of the three algorithms in Table \ref{tab8}.  Algorithm 1  is based on different polyhedrons for $ \mathbf{E}_{7,6,5}(0,0,0) $, where the  number of vertices are $ s=8,12,24 $ separately. Table \ref{444} lists the average running time of solving the one-stage projection subproblem and the total number of iterations for the computational complexity of these algorithms.  It shows that   Algorithm 1 has fewer iterations and faster velocity  because obtaining  a projection  on the boundary of linear constraints  (to solve a  standard quadratic program) is faster than   doing that for general  constraint sets. 	 	
Besides,  the increase of the number of vertices (i.e., linear constraints) has no significant impact on the computational cost of  Algorithm 1.

\begin{table}
	\centering
	\small
	\renewcommand\arraystretch{1.5}
	\setlength\tabcolsep{14pt}
	\caption{Performance of different algorithms on ellipsoid constraints }
	\begin{tabular}{c|c|c|c}
		\hline
		\hline
		\multicolumn{2}{c|}{Algorithm}&Iteration&Time (sec)\\\hline
		\multirow{3}{1.6cm}{Algorithm 1} &$ s=8 $&236&0.082\\\cline{2-4}
		{}&$ s=12 $&248&0.085\\\cline{2-4}
		{}&$ s=24 $&257&0.086 \\\cline{2-4} 	\hline			
		\multicolumn{2}{c|}{Algorithm-SQP} &306&0.131\\ \hline
		\multicolumn{2}{c|}{Algorithm-IPM} &472& 0.237\\ \hline\hline	
	\end{tabular}
	\label{tab8}
\end{table}	
\begin{table*}
	\centering
	\small
	\caption{Real running time (min)  with different dimensions of decision variables.}
	\setlength\tabcolsep{18pt}
	\renewcommand\arraystretch{1.4}
	\begin{tabular}{c|c|c|c|c|c}
		\hline
		\hline
		\specialrule{0em}{0.2pt}{0.5pt}
		Dimensions & 	 $ n=4 $&  $ n=10$ & $ n=20 $& $n=50$ &$ n=100 $\\ \hline	
		Algorithm 1 &0.05 &0.11 &1.04 &2.45  &5.14 \\ \hline
		Algorithm-SQP&0.17&0.36 &3.09 & 8.31& 17.78 \\  \hline
		Algorithm-IPM&0.36 &0.64 &5.94 &10.51& 24.83 \\\hline
		\hline
	\end{tabular}
	\label{tab6}
\end{table*}
\begin{table*}
	\centering
	\normalsize
	\setlength\tabcolsep{15pt}
	\renewcommand\arraystretch{1.3}
	\caption{Real running time   of different algorithms over different types of graphs and  various network sizes.}  
	\begin{tabular}{llllll}
		\hline
		\hline
		\specialrule{0em}{1pt}{1pt}
		\multicolumn{1}{l}{\multirow{2}{0.5cm}{Players}}
		&\multicolumn{1}{l}{\multirow{2}{1.8cm}{Feasible set constraints}}
		&\multicolumn{1}{l}{\multirow{2}{1.8cm}{Graph types}}
		&  \multicolumn{3}{l}{Real running time (min)}
		\\
		\cline{4-6}
		\multicolumn{1}{c}{}
		&
		\multicolumn{1}{c}{}
		&  {} &  {Algorithm 1}&{Algorithm-SQP}&{Algorithm-IPM}\\ \hline
		\specialrule{0em}{1pt}{1pt}
		\multirow{3}{1.5cm}{$ N =4$} &\multirow{3}{2.5cm}{$ \mathbf{E}_{5,4,3}(0,0,0)$}&\multicolumn{1}{l}{\multirow{1}{1.8cm}{ER}}&\multirow{1}{1.5cm}{0.03} &\multirow{1}{1.5cm}{0.10}&\multirow{1}{1.5cm}{0.18}\\
		{}&  {} & \multicolumn{1}{l}{ring}&{0.02}&  {0.10} & {0.23}\\
		{} & {} & {complete}&{0.03}&  {0.09} & {0.14}\\
		\hline
		\specialrule{0em}{1pt}{1pt}
		\multirow{3}{1.5cm}{$ N =20$} &\multirow{3}{2.5cm}{$ \mathbf{E}_{9,8,7}(0,0,0)$}&\multicolumn{1}{l}{\multirow{1}{1.8cm}{ER }}&\multirow{1}{1.5cm}{0.17} &\multirow{1}{1.5cm}{0.51}&\multirow{1}{1.5cm}{1.06}\\
		{}&  {}  & \multicolumn{1}{l}{ring}&{0.19}&  {0.65} & {1.20}\\
		{}& {} & {complete}&{0.17}&  {0.51} & {0.92
		}\\
		\hline	
		\specialrule{0em}{1pt}{1pt}
		\multirow{3}{1.5cm}{$ N=50 $} &\multirow{3}{2.5cm}{$ \mathbf{E}_{14,13,12}(0,0,0)$}&\multicolumn{1}{l}{\multirow{1}{1.8cm}{ER}}&\multirow{1}{1.5cm}{0.86} &\multirow{1}{1.5cm}{1.47}&\multirow{1}{1.5cm}{2.35}\\
		{}&  {}  & \multicolumn{1}{l}{ring}&{1.43}&  {3.48} & {5.26}\\
		{}&  {}  & {complete}&{1.35}&  {3.45} & {6.07}\\
		\hline		
		\specialrule{0em}{1pt}{1pt}
		\multirow{3}{1.5cm}{$ N= 100$} &\multirow{3}{2.5cm}{$ \mathbf{E}_{23,22,21}(0,0,0)$}&\multicolumn{1}{l}{\multirow{1}{1.8cm}{ER}}&\multirow{1}{1.5cm}{3.26} &\multirow{1}{1.5cm}{8.20}&\multirow{1}{1.5cm}{15.58}\\
		{}&  {}  & \multicolumn{1}{l}{ring}&{3.67}&  {8.76} & {13.23}\\
		{} & {} & {complete}&{3.76}&  {14.69} & {24.30}\\
		\hline	
		\hline			
	\end{tabular}
	\label{tab3}
\end{table*}
Note that the complexity is mainly affected by the dimension of decision variables and the number of players. 
For further comparison, we consider  Algorithm 1, Algorithm-SQP, and Algorithm-IPM for $ \epsilon $-NE (NE) seeking under different network configurations. The  payoff functions and the aggregative terms {coincide with} (\ref{54}) and (\ref{52}). Table \ref{tab6} reflects the  real running time of these algorithms under different dimensions of decision variables. Take  $n=4,10,20,50,100 $.  Here  $ \Omega_{i} $ is  a high-dimensional ball  $ \mathbf{B}_{r}(q)$ in the corresponding spaces.
On the other hand, Table \ref{tab3} reflects the real running time of these algorithms under directed ring graphs, undirected complete graphs, and
Erd\H{o}s-R\'{e}nyi (ER) graphs with various network sizes, respectively.   Take  $ \mathcal{I}=4,20,50,100 $.   $ \Omega_{i} $ is a corresponding ellipsoid ball in the three-dimensional space. Numerical results in both Table \ref{tab6} and Table \ref{tab3} show that Algorithm 1 achieves a faster convergence speed than Algorithm-SQP and Algorithm-IPM. Moreover, with the expansion of the network size and the range of set constraints,  our algorithm significantly reduces the computational cost.

	\section{Conclusion}\label{s7}
A distributed   approximate algorithm has been proposed for NE seeking of aggregative games, with the players' actions  constrained by local constraint sets  and a weight-balanced network digraph. By employing  inscribed polyhedrons to approximate players' local feasible sets, the  projection operation  has been transformed  into a  standard quadratic program. The equilibrium point of the algorithm has been  proved to be the $ \epsilon $-NE  of the original game, and the exponential convergence of the algorithm has been guaranteed. Moreover,   an upper bound of the value of $ \epsilon $  has been estimated by analyzing a  perturbed system. Finally, the computational efficiency and the approximation accuracy of our algorithm have  been illustrated by  numerical examples.
\appendices
\renewcommand{\thesection}{\Alph{section}.}
\renewcommand\thefigure{\Alph{section}\arabic{figure}}

	\section{Proof of Lemma \ref{l1}}\label{a11}
With $\mathcal{D}_{s_{1}}^{1}$ and $\mathcal{D}_{s_{2}}^{2}$ defined in (\ref{s1}) and (\ref{s2}) as two profiles of
inscribed polyhedrons of $ \Omega $, 	
we assume that $ W_{0} $ consists of the vertices constructing the hyperplanes together with $ w_{0} $ in $ \mathcal{D}_{s_{2}}^{2} $.
 Denote  $ \mathcal{M}_{0} $ as a hyperplane constructed by vertices in $ W_{0} $  without loss of generality. Denote $ \eta = H(\mathcal{D}_{s_{1}}^{1},\mathcal{D}_{s_{2}}^{2})  $,  and $ v_{0} $ as the projection point of $ w_{0} $ on $ \mathcal{M}_{0} $. Since $ \mathcal{D}_{s_{2}}^{2} $ is convex and $ w_{0} $ is on the boundary of the convex set $ \Omega $,
$$
\begin{aligned}
	&\eta = dist(\mathcal{D}_{s_{1}}^{1}, w_{0})=\!\inf \limits_{v\in \mathcal{M}_{0}}\!\|w_{0}- v\| =\|w_{0}-v_{0}\|.\\
\end{aligned}
$$
Denote $ k_{0} $ as the projection point of $ w_{0} $ on the relative boundary of $ \mathcal{M}_{0} $. Then
$$
\left\|w_{0}-k_{0}\right\|=\!\inf _{k \in r b d\left(\mathcal{M}_{0}\right)}\!\left\|w_{0}-k\right\|.
$$
Because the normalized vectors $ B^{1}_{i} $(or $ B^{2}_{i} $)  represent the normal vectors of hyperplanes enclosing $ \mathcal{D}_{s_{1}}^{1} $(or $ \mathcal{D}_{s_{2}}^{2} $), as defined in (\ref{bb}), we only need to investigate the difference between $ B^{1}_{p_{1}} $ and the last $ p_{2}-p_{1}+1 $ rows of $ B^{2} $. 	

Note that the dimension of each hyperplane is $n-1$. By the definition of the gap metric in \cite{stewart1990matrix} and \cite{qiu2005unitarily}, the angle between two hyperplanes uniquely equals to that between their normal vectors. Then there exists a derived angular metric $\psi$ and a corresponding scalar $\tau_i\in[0,\pi/2)$ for $p_{1}\leq i \leq p_{2}$  such that
$$\operatorname{sin}\tau_{i}=\psi(B^{2}_{i},B^{1}_{p_{1}})=\frac{\eta}{\|w_{0}-k_{0}\|}.$$
On this basis, we investigate  the plane containing the vectors   $  \overrightarrow{v_{0}w_{0}} $  and  $  \overrightarrow{k_{0}w_{0}} $.
We can always  construct a circular arc through the point $ w_{0} $ on this plane, where it satisfies the following conditions,
\begin{enumerate}[(a)]
	\item its center falls on the vector containing the  points $  w_{0}  $ and $ v_{0} $;
	\item all its ending points are on the hyperplane $ \mathcal{M}_{0} $, and in the interior of $ \Omega $;
	\item its diameter is larger than $ \eta $.
\end{enumerate}
Particularly, when such an arc is constructed, its curvature is also determined and does not change with the decrease of $ \eta $ in the sequel, since it is only dependent on the relative location of point $ w_{0} $ on the boundary of $ \Omega_{i} $. Denote this  curvature by $ \gamma_{i} $, and one of the ending points of this arc  by $ r_{0} $. Similarly, by the gap metric in \cite{stewart1990matrix} and  \cite{qiu2005unitarily}, denote the angle between the vector $  \overrightarrow{r_{0}w_{0}} $   and the hyperplane $ \mathcal{M}_{0} $ by $ \alpha_{i}\in[0,\pi/2) $. Accordingly, with the  angular metric $\psi$,  $\operatorname{sin} \alpha_i=\psi(\overrightarrow{r_{0}w_{0}}, \mathcal M_0)$, which eventually leads to
\begin{equation*}\label{e123}
	\operatorname{cos}\alpha_{i} = \frac{\left< w_{0}-r_{0}, v_{0}-r_{0} \right>}{\|w_{0}-r_{0}\|\cdot \|v_{0}-r_{0}\|},
\end{equation*}
where $ v_{0} $, $ r_{0} $, and $ w_{0} $ are certain points in the high-dimensional space. Furthermore,	
$$
\begin{aligned}
	\operatorname{tan}\alpha_{i}
	&=\frac{\eta}{\sqrt{\left({1/{\gamma_{i}}}\right)^{2}-\left(1/{\gamma_{i}}-\eta\right)^{2}}}=\frac{1}{\sqrt{\frac{2}{\eta\gamma_{i}}-1}}.
\end{aligned}
$$
Since $  \overrightarrow{v_{0}r_{0}} $ and $  \overrightarrow{v_{0}k_{0}} $ are collinear vectors, with
$ \|v_{0}-r_{0}\|\leq  \| v_{0}-k_{0}\|  $, 
there should be  $ \tau_{i}\leq \alpha_{i} $. Recalling  $\alpha_{i} \leq \operatorname{tan}\alpha_{i}$,  we have
\begin{equation}\label{aad}
	\tau_{i}\leq\alpha_{i}\leq \frac{1}{\sqrt{\frac{2}{\eta\gamma_{i}}-1}}.
\end{equation}		
Additionally, for $p_{1}\leq i \leq p_{2}$, there is $ P_{i}\in SO(n) $ such that $ B^{2}_{i}= B^{1}_{p_{1}}P_{i} $. From \cite[Theorem 2.21]{Hall2003Lie}, 
$$ P_{i}= I+\tau_{i}\mathcal{V}_{i}+o(\tau_{i}), \quad \mathcal{V}_{i}\in \mathfrak{g}(SO(n)),$$
where $ \mathcal{V}_{i} $ is a constant matrix  and $ \mathfrak{g}(\cdot) $ represents its Lie algebra.	
Since $ B^{1} $ and $ B^{2} $ are normalized rows,
$$\begin{aligned}
	\left\|B_{i}^{2}-B_{p_{1}}^{1}\right\| &=\left\|(P_{i}-I)B^{1}_{p_{1}}\right\|
	&=\left\|\tau_{i}\mathcal{V}_{i}+o(\tau_{i})\right\|.\\
\end{aligned}$$	
Note  that
$
\lim \limits_{ \eta\rightarrow 0}
\frac{1}{\sqrt{\frac{2}{\eta\gamma_{i}}-1}}
=0$.
Together with (\ref{aad}), we have
$$\begin{aligned}
	\left\|B_{i}^{2}-B_{p_{1}}^{1}\right\| =\left\|\tau_{i}\mathcal{V}_{i}+o(\tau_{i})\right\| \rightarrow 0, \;\; \quad\mathrm{as}\;\eta\rightarrow 0.
\end{aligned}$$
Thus, the conclusion follows.
\hfill $\square$


\section{Proof of Lemma \ref{t3}}         \label{a13}

Note that $ e(\boldsymbol{z}) $ is related to  the difference caused by projection  on the inscribed polyhedron  $ \boldsymbol{\mathcal{D}}_{s} $  and the original action set $ \boldsymbol{\Omega} $. For $ i\in\mathcal{I} $, denote $ x^{1} $ and  $ x^{2} $
as two projection points on $ \Omega_{i} $ and $ \mathcal{D}_{s_{i}}^{i} $, respectively, i.e.,
$$ \begin{aligned}
	&x^{1}= \Pi_{\Omega_{i}}(x_{i}-\beta_{1} U_{i}(x_{i}, \zeta_{i})),\\
	&x^{2}= \Pi_{\mathcal{D}_{s_{i}}^{i}}(x_{i}-\beta_{1} U_{i}(x_{i}, \zeta_{i})).
\end{aligned}$$
Then $$ \|x^{1}-x^{2}\|\!\!=\!\! \|\Pi_{\Omega_{i}}\!(x_{i}-\beta_{1} U_{i}(x_{i}, \zeta_{i})\!)-\Pi_{\mathcal{D}_{s_{i}}^{i}}\!(x_{i}-\beta_{1} U_{i}(x_{i}, \zeta_{i})\!)\|. $$

Recalling the definition of the inscribed polyhedron, $x^2$ is the point projected onto a hyperplane basically, where
we can construct a vector perpendicular to this  hyperplane and passing through  $x^2$. Denote the intersection point between this vector and   the boundary of $ \Omega_{i} $ by  $ x^{3} $. Eventually, $ x^{1} $, $ x^{2} $, and $ x^{3} $ form a triangle. Therefore,
$$
\|x^{1}-x^{2}\|<\|x^{1}-x^{3}\|+\|x^{2}-x^{3}\|.
$$
Recall the definition of $ \boldsymbol{H} $, for the $i$th player,
\begin{equation*}
	\begin{aligned}
		h_{i}&=\| x_{\diamond}-y_{\diamond}\| \\
		&=\max \left\{\sup \limits_{x \in \mathcal{D}_{s_{i}}^{i}} \operatorname{dist}(x, \Omega_{i}), \sup \limits_{y \in \Omega_{i}} \operatorname{dist}(y, \mathcal{D}_{s_{i}}^{i})\right\}.
	\end{aligned}
\end{equation*}
Accordingly,   we investigate  a plane containing the vector   $  \overrightarrow{x_{\diamond}y_{\diamond}} $. We try to find a curvature related with $ \Omega_{i} $  and then
construct a   piece of a  circular arc with this  curvature  on this plane. Similar to Lemma \ref{l1}, the constructed arc needs to satisfy
\begin{enumerate}[(a)] 	
	\item it  passes through the  point  $ y_{\diamond} $  and its center  falls on the vector
	containing the  points of $  x_{\diamond}$ and  $y_{\diamond} $;
	\item its ending points
	are on the hyperplane  perpendicular to the vector $  \overrightarrow{x_{\diamond}y_{\diamond} } $;
	\item its ending points are outside $ \Omega_{i} $.
\end{enumerate}
Since the  found curvature is related with $ \Omega_{i} $ rather than any approximation information,  its value  can be regarded as a constant. It is obvious that  there always exists such an arc.   Denote this curvature by $ \nu_{i} $, i.e., the circular radius by $ 1/\nu_{i} $, the center point of the arc by $ c_{0} $, the angle corresponding to the arc by  $ \theta_{i} $,  an ending point of the arc by $ d_{0} $, and
the length of the   arc  by $l_{i}$. Clearly,   $ l_{i}= \theta_{i}/{\nu_{i}}$. Then, recalling the gap metric and the derived angular metric in \cite{stewart1990matrix} and  \cite{qiu2005unitarily}, we  have
\begin{equation*}\label{e18}	
	\begin{aligned}
		&\theta_{i}=2\operatorname{arccos}\frac{\|c_{0}-x_{\diamond} \| }{\| d_{0}-c_{0}\|}=2\operatorname{arccos}(1-\nu_{i}h_{i}).
	\end{aligned}
\end{equation*}
Moreover,  $\|x^{2}-x^{3}\|$ can be bounded by the Hausdorff distance $ h_{i} $ and   $ \|x^{1}-x^{3}\|  $ can be bounded by the length of arc  $ l_{i} $ intuitively. 
Consequently, the bound of the difference between the projection dynamics on $ \mathcal{D}_{s_{i}}^{i} $ and $ \Omega_{i} $ of the $ i $th player can be expressed by
\begin{equation*}\label{e14}	
	\begin{aligned}
		&\|\Pi_{\Omega_{i}}(x_{i}-\beta_{1} U_{i}(x_{i}, \zeta_{i}))-\Pi_{\mathcal{D}_{s_{i}}^{i}}(x_{i}-\beta_{1} U_{i}(x_{i}, \zeta_{i}))\|\\
		&\leq l_{i}+h_{i}\\
		&=\frac{2}{ \nu_{i}}\operatorname{arccos}(1-\nu_{i}h_{i})+h_{i}.
	\end{aligned}
\end{equation*}
The analysis of other players is similar to that of player $ i $.
To sum up,  $ \|e(\boldsymbol{z})\| $ can be bounded by $\delta(\boldsymbol{H})$, that is,
\begin{equation*}\label{e19}	
	\begin{aligned}
		\|e(\boldsymbol{z})\|&\leq(1+\|\nabla \boldsymbol{q}(\boldsymbol{x})-\mathbf{1}_{N} \otimes \nabla \mathcal{Q}(\boldsymbol{x})\|) \cdot\\	
		&\quad\,\|\Pi_{\boldsymbol{\mathcal{D}}_{s}}(\boldsymbol{x}-\beta_{1} U(\boldsymbol{x}, \boldsymbol{\zeta}))-\Pi_{\boldsymbol{\Omega}}(\boldsymbol{x}-\beta_{1} U(\boldsymbol{x}, \boldsymbol{\zeta}))\| \\
		&=(1+c_{3})\sqrt{ \sum_{i=1}^{N}\left(\frac{2}{ \nu_{i}}\operatorname{arccos}(1-\nu_{i}h_{i})+h_{i}\right)^{2}}\\
		&\triangleq\delta(\boldsymbol{H}),
	\end{aligned}
\end{equation*}
where $ c_{3} $ is a  Lipschitz constant of $ q_{i} $ for $ i\in \mathcal{I} $. 	This yields the conclusion.
\hfill $\square$
\section{Proof of Lemma \ref{l7}}  \label{a5}
Take $  V_{\boldsymbol{\Omega}}(\boldsymbol{z}) $  as a Lyapunov  function of (\ref{e11}) that satisfies (\ref{e22}). Then the derivative of 	$ V_{\boldsymbol{\Omega}}(\boldsymbol{z}) $ along the trajectories of (\ref{e12}) satisfies
$$
\begin{aligned}
	\dot{V}_{\boldsymbol{\Omega}}(\boldsymbol{z}) & \leq-a_{3}\|\boldsymbol{z}-\boldsymbol{z}^{*}\|^{2}+\left\|\frac{\partial V_{\boldsymbol{\Omega}}}{\partial \boldsymbol{z}}\right\|\|e( \boldsymbol{z})\| \\
	& \leq-a_{3}\|\boldsymbol{z}-\boldsymbol{z}^{*}\|^{2}+a_{4} \delta\|\boldsymbol{z}-\boldsymbol{z}^{*}\|, \\
\end{aligned}
$$	
where $ \delta=\delta(\boldsymbol{H}) $, $ \boldsymbol{z}^{*}=\boldsymbol{z}^{*}(\boldsymbol{\Omega}) $.
For a positive constant $ \mu<1 $ and $ \|\boldsymbol{z}-\boldsymbol{z}^{*}\| \geq \delta a_{4} / \mu a_{3} $, it satisfies
$$			
\begin{aligned}
	\dot{V}_{\boldsymbol{\Omega}}(\boldsymbol{z})&\leq-(1-\mu) a_{3}\|\boldsymbol{z}-\boldsymbol{z}^{*}\|^{2}-\mu a_{3}\|\boldsymbol{z}-\boldsymbol{z}^{*}\|^{2}\\
	&+a_{4} \delta\|\boldsymbol{z}-\boldsymbol{z}^{*}\| \quad  \\
	& \leq-(1-\mu) a_{3}\|\boldsymbol{z}-\boldsymbol{z}^{*}\|^{2}.
\end{aligned}
$$ 		
Denote $ K=\delta a_{4} / \mu a_{3} $. 	
Once $ V_{\boldsymbol{\Omega}}\geq a_{2} K^{2} $, $ \|\boldsymbol{z}-\boldsymbol{z}^{*}\|\!\geq\! K $ and $ \dot{V}_{\boldsymbol{\Omega}} \leq -(1-\mu)a_{3}/a_{2}V_{\boldsymbol{\Omega}} $, which implies
$$
V_{\boldsymbol{\Omega}}(\boldsymbol{z})\leq e^{-(1-\mu)a_{3}/a_{2} (t-t_{0})}V_{\boldsymbol{\Omega}}(\boldsymbol{z}(t_{0})).
$$	
Hence,
$$
\begin{aligned}
	\|\boldsymbol{z}(t)-\boldsymbol{z}^{*}\| & \leq\left(\frac{V_{\boldsymbol{\Omega}}(\boldsymbol{z})}{a_{1}}\right)^{1 / 2}\\ &\leq\left(\frac{1}{a_{1}} e^{-(1-\mu)a_{3}/a_{2}(t-t_{0})} V_{\boldsymbol{\Omega}}\left( \boldsymbol{z}(t_{0})\right)\right)^{1 / 2} \\
	&=\sqrt{\frac{a_{2}}{a_{1}}} e^{-\omega(t-t_{0})}\left\|\boldsymbol{z}(t_{0})\right\|,
\end{aligned}
$$	
which holds over the interval $ [t_{0},t_{0}+T) $ when $ V_{\boldsymbol{\Omega}}\geq a_{2}K^{2} $. For $ t\geq t_{0}+T $, we have
$$
\|\boldsymbol{z}(t)-\boldsymbol{z}^{*}\| \leq\sqrt{\frac{V_{\boldsymbol{\Omega}}(\boldsymbol{z})}{a_{1}}}\leq\sqrt{\frac{a_{2}}{a_{1}}}K=R.
$$
This yields the conclusion.
\hfill $\square$

	\bibliographystyle{IEEEtran}
\bibliography{autosam}

\begin{thebibliography}{10}
\providecommand{\url}[1]{#1}
\csname url@samestyle\endcsname
\providecommand{\newblock}{\relax}
\providecommand{\bibinfo}[2]{#2}
\providecommand{\BIBentrySTDinterwordspacing}{\spaceskip=0pt\relax}
\providecommand{\BIBentryALTinterwordstretchfactor}{4}
\providecommand{\BIBentryALTinterwordspacing}{\spaceskip=\fontdimen2\font plus
\BIBentryALTinterwordstretchfactor\fontdimen3\font minus
  \fontdimen4\font\relax}
\providecommand{\BIBforeignlanguage}[2]{{%
\expandafter\ifx\csname l@#1\endcsname\relax
\typeout{** WARNING: IEEEtran.bst: No hyphenation pattern has been}%
\typeout{** loaded for the language `#1'. Using the pattern for}%
\typeout{** the default language instead.}%
\else
\language=\csname l@#1\endcsname
\fi
#2}}
\providecommand{\BIBdecl}{\relax}
\BIBdecl

\bibitem{ye2016game}
M.~Ye and G.~Hu, ``Game design and analysis for price-based demand response: An
  aggregate game approach,'' \emph{IEEE Transactions on Cybernetics}, vol.~47,
  no.~3, pp. 720--730, 2017.

\bibitem{nocke2018multiproduct}
V.~Nocke and N.~Schutz, ``Multiproduct-firm oligopoly: An aggregative games
  approach,'' \emph{Econometrica}, vol.~86, no.~2, pp. 523--557, 2018.

\bibitem{gharesifard2015price}
B.~Gharesifard, T.~Ba{\c{s}}ar, and A.~D. Dom{\'\i}nguez-Garc{\'\i}a,
  ``Price-based coordinated aggregation of networked distributed energy
  resources,'' \emph{IEEE Transactions on Automatic Control}, vol.~61, no.~10,
  pp. 2936--2946, 2016.

\bibitem{parise2019distributed}
F.~Parise, B.~Gentile, and J.~Lygeros, ``A distributed algorithm for
  almost-{N}ash equilibria of average aggregative games with coupling
  constraints,'' \emph{IEEE Transactions on Control of Network Systems},
  vol.~7, no.~2, pp. 770--782, 2019.

\bibitem{de2019distributed}
C.~De~Persis and S.~Grammatico, ``Distributed averaging integral {N}ash
  equilibrium seeking on networks,'' \emph{Automatica}, vol. 110, p. 108548,
  2019.

\bibitem{yi2019asynchronous}
P.~Yi and L.~Pavel, ``Asynchronous distributed algorithms for seeking
  generalized {N}ash equilibria under full and partial-decision information,''
  \emph{IEEE Transactions on Cybernetics}, vol.~50, no.~6, pp. 2514--2526,
  2019.

\bibitem{koshal2016distributed}
J.~Koshal, A.~Nedi{\'c}, and U.~V. Shanbhag, ``Distributed algorithms for
  aggregative games on graphs,'' \emph{Operations Research}, vol.~64, no.~3,
  pp. 680--704, 2016.

\bibitem{paccagnan2016distributed}
D.~Paccagnan, B.~Gentile, F.~Parise, M.~Kamgarpour, and J.~Lygeros,
  ``Distributed computation of generalized {N}ash equilibria in quadratic
  aggregative games with affine coupling constraints,'' in \emph{2016 {IEEE}
  55th {C}onference on {D}ecision and {C}ontrol ({CDC})}.\hskip 1em plus 0.5em
  minus 0.4em\relax IEEE, 2016, pp. 6123--6128.

\bibitem{liang2017distributed}
S.~Liang, P.~Yi, and Y.~Hong, ``Distributed {N}ash equilibrium seeking for
  aggregative games with coupled constraints,'' \emph{Automatica}, vol.~85, pp.
  179--185, 2017.

\bibitem{lei2020distributed}
J.~Lei, U.~V. Shanbhag, and J.~Chen, ``Distributed computation of {N}ash
  equilibria for monotone aggregative games via iterative regularization,'' in
  \emph{2020 59th IEEE Conference on Decision and Control (CDC)}.\hskip 1em
  plus 0.5em minus 0.4em\relax IEEE, 2020, pp. 2285--2290.

\bibitem{belgioioso2020distributed}
G.~Belgioioso, A.~Nedich, and S.~Grammatico, ``Distributed generalized {N}ash
  equilibrium seeking in aggregative games on time-varying networks,''
  \emph{IEEE Transactions on Automatic Control}, vol.~66, no.~5, pp.
  2061--2075, 2020.

\bibitem{boggs1995sequential}
P.~T. Boggs and J.~W. Tolle, ``Sequential quadratic programming,'' \emph{Acta
  {N}umerica}, vol.~4, no.~1, pp. 1--51, 1995.

\bibitem{nesterov1994interior}
Y.~Nesterov and A.~Nemirovskii, \emph{Interior-{P}oint {P}olynomial
  {A}lgorithms in {C}onvex {P}rogramming}.\hskip 1em plus 0.5em minus
  0.4em\relax SIAM, 1994.

\bibitem{fortin2000augmented}
M.~Fortin and R.~Glowinski, \emph{Augmented {L}agrangian {M}ethods:
  {A}pplications to the {N}umerical {S}olution of {B}oundary-{V}alue
  {P}roblems}.\hskip 1em plus 0.5em minus 0.4em\relax Elsevier, 2000.

\bibitem{li2019hyperplane}
Y.~Li, H.-L. Liu, and E.~D. Goodman, ``Hyperplane-approximation-based method
  for many-objective optimization problems with redundant objectives,''
  \emph{Evolutionary Computation}, vol.~27, no.~2, pp. 313--344, 2019.

\bibitem{jurie2002hyperplane}
F.~Jurie and M.~Dhome, ``Hyperplane approximation for template matching,''
  \emph{IEEE Transactions on Pattern Analysis and Machine Intelligence},
  vol.~24, no.~7, pp. 996--1000, 2002.

\bibitem{vincent2001k}
P.~Vincent and Y.~Bengio, ``K-local hyperplane and convex distance nearest
  neighbor algorithms,'' in \emph{NeurIPS}, vol.~14, 2001, pp. 985--992.

\bibitem{facchinei2007finite}
F.~Facchinei and J.-S. Pang, \emph{Finite-{D}imensional {V}ariational
  {I}nequalities and {C}omplementarity {P}roblems}.\hskip 1em plus 0.5em minus
  0.4em\relax Springer Science \& Business Media, 2007.

\bibitem{bullo2009distributed}
F.~Bullo, J.~Cort{\'e}s, and S.~Mart{\'\i}nez, \emph{Distributed {C}ontrol of
  {R}obotic {N}etworks}, ser. Applied Mathematics.\hskip 1em plus 0.5em minus
  0.4em\relax Princeton University Press, 2009.

\bibitem{yi2019operator}
P.~Yi and L.~Pavel, ``An operator splitting approach for distributed
  generalized {N}ash equilibria computation,'' \emph{Automatica}, vol. 102, pp.
  111--121, 2019.

\bibitem{zhang2019distributed11}
Y.~Zhang, S.~Liang, X.~Wang, and H.~Ji, ``Distributed {N}ash equilibrium
  seeking for aggregative games with nonlinear dynamics under external
  disturbances,'' \emph{IEEE Transactions on Cybernetics}, vol.~50, no.~12, pp.
  4876--4885, 2019.

\bibitem{lu2018distributed}
K.~Lu, G.~Jing, and L.~Wang, ``Distributed algorithms for searching generalized
  {N}ash equilibrium of noncooperative games,'' \emph{IEEE Transactions on
  Cybernetics}, vol.~49, no.~6, pp. 2362--2371, 2019.

\bibitem{liang2019distributed}
S.~Liang, P.~Yi, Y.~Hong, and K.~Peng, ``Distributed {N}ash equilibrium seeking
  for aggregative games via a small-gain approach,'' \emph{arXiv preprint
  arXiv: 1911.06458}, 2019.

\bibitem{deng2019distributed}
Z.~Deng, X.~Nian, and C.~Hu, ``Distributed algorithm design for nonsmooth
  resource allocation problems,'' \emph{IEEE Transactions on Cybernetics},
  vol.~50, no.~7, pp. 3208--3217, 2019.

\bibitem{tang2018distributed}
Y.~Tang, ``Distributed optimal steady-state regulation for high-order
  multiagent systems with external disturbances,'' \emph{IEEE Transactions on
  Systems, Man, and Cybernetics: Systems}, vol.~50, no.~11, pp. 4828--4835,
  2018.

\bibitem{liang2020distributed}
S.~Liang, X.~Zeng, G.~Chen, and Y.~Hong, ``Distributed sub-optimal resource
  allocation via a projected form of singular perturbation,''
  \emph{Automatica}, vol. 121, p. 109180, 2020.

\bibitem{deng2018distributed}
Z.~Deng and X.~Nian, ``Distributed generalized {N}ash equilibrium seeking
  algorithm design for aggregative games over weight-balanced digraphs,''
  \emph{IEEE Transactions on Neural Networks and Learning Systems}, vol.~30,
  no.~3, pp. 695--706, 2018.

\bibitem{dudley1974metric}
R.~M. Dudley, ``Metric entropy of some classes of sets with differentiable
  boundaries,'' \emph{Journal of Approximation Theory}, vol.~10, no.~3, pp.
  227--236, 1974.

\bibitem{bronshtein1976varepsilon}
E.~M. Bronshtein, ``$\varepsilon$-entropy of convex sets and functions,''
  \emph{Siberian Mathematical Journal}, vol.~17, no.~3, pp. 393--398, 1976.

\bibitem{lou2014approximate}
Y.~Lou, G.~Shi, K.~H. Johansson, and Y.~Hong, ``Approximate projected consensus
  for convex intersection computation: Convergence analysis and critical error
  angle,'' \emph{IEEE Transactions on Automatic Control}, vol.~59, no.~7, pp.
  1722--1736, 2014.

\bibitem{chen2021distributed}
G.~Chen, Y.~Ming, Y.~Hong, and P.~Yi, ``Distributed algorithm for
  $\varepsilon$-generalized {N}ash equilibria with uncertain coupled
  constraints,'' \emph{Automatica}, vol. 123, p. 109313, 2021.

\bibitem{2008Approximation}
E.~M. Bronstein, ``Approximation of convex sets by polytopes,'' \emph{Journal
  of Mathematical Sciences}, vol. 153, no.~6, pp. 727--762, 2008.

\bibitem{bushenkov1985iteration}
V.~A. Bushenkov, ``An iteration method of constructing orthogonal projections
  of convex polyhedral sets,'' \emph{USSR Computational Mathematics and
  Mathematical Physics}, vol.~25, no.~5, pp. 1--5, 1985.

\bibitem{dzholdybaeva1992numerical}
S.~Dzholdybaeva and G.~Kamenev, ``Numerical study of the effectiveness of the
  algorithm of approximation of convex bodies by polyhedrons,'' \emph{Zh.
  Vychisl. Mat. Mat. Fiz}, vol.~32, pp. 857--866, 1992.

\bibitem{Kamenev1993The}
{G. K. Kamenev}, ``The efficiency of {H}ausdorff algorithms for approximating
  convex bodies by polytopes,'' \emph{Computational Mathematics and
  Mathematical Physics}, vol.~33, no.~5, pp. 709--716, 1993.

\bibitem{cherukuri2016initialization}
A.~Cherukuri and J.~Cort{\'e}s, ``Initialization-free distributed coordination
  for economic dispatch under varying loads and generator commitment,''
  \emph{Automatica}, vol.~74, pp. 183--193, 2016.

\bibitem{2010The}
H.~Markowitz, ``The optimization of a quadratic function subject to linear
  constraints,'' \emph{Naval Research Logistics}, vol.~3, no. 1-2, pp.
  111--133, 1956.

\bibitem{vavasis1991nonlinear}
S.~A. Vavasis, \emph{Nonlinear {O}ptimization: {C}omplexity {I}ssues}.\hskip
  1em plus 0.5em minus 0.4em\relax Oxford University Press, Inc., 1991.

\bibitem{ye2014new}
Y.~Ye, C.~A. Floudas, and P.~M. Pardalos, ``A new complexity result on
  minimization of a quadratic function with a sphere constraint,'' in
  \emph{Recent Advances in Global Optimization}.\hskip 1em plus 0.5em minus
  0.4em\relax Princeton University Press, 1991, pp. 19--31.

\bibitem{monteiro1989interior}
R.~D. Monteiro and I.~Adler, ``Interior path following primal-dual algorithms.
  part ii: Convex quadratic programming,'' \emph{Mathematical Programming},
  vol.~44, no.~1, pp. 43--66, 1989.

\bibitem{calamai1987projected}
P.~H. Calamai and J.~J. Mor{\'e}, ``Projected gradient methods for linearly
  constrained problems,'' \emph{Mathematical {P}rogramming}, vol.~39, no.~1,
  pp. 93--116, 1987.

\bibitem{khalil2002nonlinear}
H.~K. Khalil, \emph{Nonlinear {S}ystems}, 3rd~ed.\hskip 1em plus 0.5em minus
  0.4em\relax New Jersey: Prentice Hall, 2002.

\bibitem{stewart1990matrix}
G.~W. Stewart and J.-G. Sun, \emph{Matrix {P}erturbation {T}heory}.\hskip 1em
  plus 0.5em minus 0.4em\relax Academic Press, Boston, 1990.

\bibitem{qiu2005unitarily}
L.~Qiu, Y.~Zhang, and C.-K. Li, ``Unitarily invariant metrics on the grassmann
  space,'' \emph{SIAM {J}ournal on {M}atrix {A}nalysis and {A}pplications},
  vol.~27, no.~2, pp. 507--531, 2005.

\bibitem{Hall2003Lie}
B.~C. Hall, \emph{Lie Groups, {L}ie Algebras, and {R}epresentations}.\hskip 1em
  plus 0.5em minus 0.4em\relax New York: Springer-Verlag, 2003.

\end{thebibliography}

%








\end{document}